\documentclass[12pt]{amsart}
\usepackage{amsmath,amssymb,amsthm,color, euscript, enumerate,comment}
\usepackage{dsfont}
\usepackage{longtable}
\newcommand{\rank}{{\rm rank}\, }
\renewcommand{\l}{\left}
\renewcommand{\r}{\right}

\newcommand{\maru}[1]{{\ooalign{\hfil#1\/\hfil\crcr
\raise.167ex\hbox{\mathhexbox20D}}}}

\newcommand{\ruby}[2]{%
 \leavevmode
 \setbox0=\hbox{#1}%
 \setbox1=\hbox{\tiny #2}%
 \ifdim\wd0>\wd1 \dimen0=\wd0 \end{lemma}se \dimen0=\wd1 \fi
 \hbox{%
   \kanjiskip=0pt plus 2fil
   \xkanjiskip=0pt plus 2fil
   \vbox{%
     \hbox to \dimen0{%
       \tiny \hfil#2\hfil}%
     \nointerlineskip
     \hbox to \dimen0{\mathstrut\hfil#1\hfil}}}}

\newcommand{\la}{\langle}
\newcommand{\ra}{\rangle}

\newcommand{\Z}{\mathbb{Z}}
\newcommand{\C}{\mathbb{C}}
\newcommand{\R}{\mathbb{R}}

\newcommand{\Q}{\mathbb{Q}}

\newcommand{\g}{\mathfrak{g}}

\newcommand{\h}{\mathfrak{h}}
\newcommand{\F}{\mathbb{F}}

\newcommand{\End}{\mathrm{End}}

\newcommand{\aut}{\mathrm{Aut}\,}
\newcommand{\Aut}{\mathrm{Aut}\,}

\newcommand{\tr}{\mathrm{tr}}

\newcommand{\ad}{\mathrm{ad}}

\newcommand{\Com}{\mathrm{Com}}
\newcommand{\Span}{\mathrm{Span}}

\newcommand{\vacuum}{\mathrm{1\hspace{-3.2pt}l}}
\newcommand{\vac}{\vacuum}

\newcommand{\irr}{\mathrm{Irr}}

\newcommand{\allzero}{\mathbf{0}}
\makeatletter \@addtoreset{equation}{section}

\theoremstyle{plain}
\newtheorem{maintheorem}{Main Theorem}
\newtheorem{maincoro}{Main Corollary} 
\newtheorem{theorem}{Theorem}[section]
\newtheorem{proposition}[theorem]{Proposition}
\newtheorem{lemma}[theorem]{Lemma}
\newtheorem{corollary}[theorem]{Corollary}

\theoremstyle{definition}
\newtheorem{definition}[theorem]{Definition}

\theoremstyle{remark}
\newtheorem{remark}[theorem]{Remark}
\numberwithin{equation}{section}

\title[Extra automorphisms of cyclic orbifolds of lattice VOAs]{Extra automorphisms of cyclic orbifolds of\\ lattice vertex operator algebras}
\author{Ching Hung Lam} %
  \address[C. H. Lam] {Institute of Mathematics, Academia Sinica, Taipei 10617, Taiwan} 
  \email{chlam@math.sinica.edu.tw}
\author[H. Shimakura]{Hiroki Shimakura}%
\address[H. Shimakura]{Department of Applied Mathematics, Faculty of Sciences
Fukuoka University, Fukuoka 814-0180, Japan }%
\email {shimakura@fukuoka-u.ac.jp}%
\date{}
\thanks{C.H.Lam was partially supported by a research grant AS-IA-107-M02 of Academia Sinica  and MOST grant  107-2115-M-001-003-MY3  of Taiwan}
\thanks{H.\ Shimakura was partially supported by JSPS KAKENHI Grant Numbers JP19KK0065 and JP20K03505.}

\newcommand{\sfr}[2]{\leavevmode\kern-.1em
  \raise.5ex\hbox{\the\scriptfont0 #1}\kern-.1em
  /\kern-.15em\lower.25ex\hbox{\the\scriptfont0 #2}}

\begin{document}

\begin{abstract}
In this article, we study automorphisms of the cyclic orbifold of a vertex operator algebra associated with a rootless even lattice for a lift of a fixed-point free isometry of odd prime order $p$.
We prove that such a cyclic orbifold contains extra automorphisms, not induced from automorphisms of the lattice vertex operator algebra, if and only if the rootless even lattice can be constructed by Construction B from a code over $\mathbb{Z}_p$ or is isometric to the coinvariant lattice of the Leech lattice associated with a certain isometry of order $p$.
\end{abstract}

\maketitle

%\tableofcontents

\section{Introduction}
The orbifold of a vertex operator algebra (VOA) for  an automorphism group is the fixed-point subVOA.
This is a standard method to construct new VOAs from known VOAs.
Recently, it has been proved in \cite{CM,Mc22, Mi} that a (cyclic) orbifold has ``nice'' properties, such as rationality and $C_2$-cofiniteness, if the original VOA has the same properties.
It is natural to ask if all automorphisms of the orbifold can be obtained from those of the original VOA; more precisely, for a VOA $V$ and an automorphism group $G$, the automorphism group $\Aut(V^G)$ of the orbifold $V^G$ is a quotient of the normalizer $N_{\Aut(V)}(G)$ of $G$ in $\Aut(V)$ or not.
We say that an automorphism of $V^{G}$ is \emph{extra} if it cannot be obtained from $N_{\Aut(V)}(G)$. Therefore, the key question is to determine when the orbifold has extra automorphisms.
In order to establish a general theory, one has to analyze many examples.

One of the fundamental examples of VOAs is a lattice VOA $V_L$ associated with an even lattice $L$.
It is known \cite{FLM} that the isometry group $O(L)$ of $L$ lifts to the subgroup $O(\hat{L})$ of $\Aut(V_L)$; let $\hat{g}\in O(\hat{L})$ be a (standard) lift of an isometry $g\in O(L)$ (cf. Section \ref{sec:2.3}).  
If $L$ has norm $2$ vectors, called roots, or $g$ has fixed-points, then the weight $1$ subspace of the cyclic orbifold $V_L^{\hat{g}}$ is non-trivial and $\aut(V_L^{\hat{g}})$ is infinite.  
Hence, we focus on the case where $L$ is rootless and $g$ is fixed-point free.
Indeed, many interesting finite groups, such as the orthogonal groups $O^+(10,2)$ and $\Omega^-_8(3){:}2$, are obtained as the automorphism groups of cyclic orbifolds of lattice VOAs under this setting (\cite{Gr1,Sh04,CLS}).  
Some other examples are studied in \cite{Lam18b,LamCFT,BLS}, which shows that the cyclic orbifold VOA $V_{\Lambda_h}^{\hat{h}}$ has extra automorphisms  for a certain isometry  $h$ of the Leech lattice $\Lambda$, where $\Lambda_h$ denotes the coinvariant lattice of $\Lambda$ associated with $h$.

In this article, we deal with the case where $g$ is a fixed-point free isometry of odd prime order and classify all rootless even lattices such that $V_L^{\hat{g}}$ has extra automorphisms. The main idea is to study the irreducible modules of $V_L^{\hat{g}}$ which have similar properties as the module $V_L(1)$, where  $V_{L}(j)=\{ v\in V_{L}\mid \hat{g}(v)= e^{2\pi \sqrt{-1} j/|\hat{g}|} v\}$, $0\leq j \leq |\hat{g}|-1$. 
 
Recall that the case where the order of $g$ is $2$, that is, $g$ is the $-1$-isometry, has been studied in \cite{Sh04};
$V_L^{\hat{g}}$ has extra automorphisms if and only if $L$ can be constructed by Construction B from a doubly even binary code (see \cite[Chapter 7.5]{CS} for the details about Construction B).
Moreover, the $\tau$-conjugate $V_L(1)\circ \tau$ is of twisted type  for some $\tau\in\Aut(V_L^{\hat{g}})$ if and only if $L$ is isometric to $\sqrt2E_8$ or the Barnes-Wall lattice of rank $16$. 
Please refer to Definition \ref{Mconj} and Section \ref{S:tw} for the definitions of $\tau$-conjugates and modules of twisted type. Notice that both lattices can be constructed by Construction B.

We consider a generalized Construction B in Section \ref{Sec:ConstAB}, 
which is different from \cite[Chapter 7.5]{CS} if $p>2$,
 and generalize the results of \cite{Sh04} in Sections 5, 6 and 7. Our main theorem is as follows:

\begin{maintheorem}
Let $L$ be a rootless even lattice of rank $m$ and $p$ a prime number.
Let $g$ be a fixed-point free isometry of $L$ of order $p$ and let $\hat{g}$ be a (standard) lift of $g$.
Then $V_L^{\hat{g}}$ has extra automorphisms if and only if one of the following holds:
\begin{enumerate}[{\rm (i)}]
\item $p=2$ and $L$ can be constructed by Construction B from a doubly even binary code of length $m$. In this case, $g$ is the $-1$-isometry. 
\item $p$ is an odd prime and $L$ can be constructed by Construction B from a self-orthogonal code of length $m/(p-1)$ over $\Z_p$ whose dual code contains a codeword $e$ of Hamming weight $m/(p-1)$; $g$ is given by $g_{\Delta,e}$ as defined in \eqref{Eq:gre}. 
\item $p=11$ and $L$ is isometric to the coinvariant lattice $\Lambda_{11A}$ of the Leech lattice associated with the conjugacy class $11A$ of $O(\Lambda)$; $g$ is given by the restriction of the $11A$ element on $\Lambda_{11A}$. 
\item $p=23$ and $L$ is isometric to the coinvariant lattice $\Lambda_{23A}$ of the Leech lattice associated with the conjugacy class $23A$ of $O(\Lambda)$; $g$ is given by the restriction of the $23A$ element on $\Lambda_{23A}$. 
\end{enumerate}
\end{maintheorem}

As a corollary of Main Theorem 1, we obtain the following:

\begin{maincoro} Let $L$, $p$ and $g$  be as in  Main Theorem 1.
If $L$ does not satisfy (i), (ii), (iii) and (iv) in Main Theorem 1, then $\aut(V_L^{\hat{g}})\cong N_{\Aut(V_L)}(\langle\hat{g}\rangle)/\langle\hat{g}\rangle$.
\end{maincoro}

We also classify rootless even lattices $L$ such that the $\tau$-conjugate $V_L(1)\circ\tau$ is of twisted type for some $\tau\in\Aut(V_L^{\hat{g}})$ in Theorem \ref{T:(II)}.
In particular, we obtain the following:

\begin{maincoro} Let $L$, $p$ and $g$  be as in  Main Theorem 1. 
Then the following are equivalent:
\begin{itemize}
\item There exists an automorphism $\tau$ of $V_L^{\hat{g}}$ such that $V_L(1)\circ\tau$ is of twisted type;
\item $L$ is isometric to the coinvariant lattice $\Lambda_{pX}$ of the Leech lattice $\Lambda$ associated with the conjugacy class $pX\in\{2A,-2A,3B,3C,5B,5C,7B,11A,23A\}$ 
and $g$ is given by the restriction of the $pX$ element on $\Lambda_{pX}$.
\end{itemize}
\end{maincoro}

Our results consist of existence part of extra automorphisms of $V_L^{\hat{g}}$ and classification part of rootless even lattices $L$ such that $V_L^{\hat{g}}$ has extra automorphisms.

For the existence, we explicitly construct an extra automorphism of $V_L^{\hat{g}}$ if $L$ can be constructed by Construction B from a code over $\Z_p$ (Corollary \ref{C:extra}).
More generally, we deal with the case where $L$ is constructed from a subgroup of $\bigoplus_{i=1}^t \Z_{k_i}$ (Theorem \ref{thm:extra}).
The main idea is to generalize the triality automorphisms associated with the root lattice of type $A_1$ in \cite[Chapter 11]{FLM} to other root lattices of type $A_{k-1}$ $(k\ge2)$. 
In addition, we prove that $V_L^{\hat{g}}$ has extra automorphisms if $L$ is the coinvariant lattice $\Lambda_h$ of the Leech lattice $\Lambda$ associated with some $h\in O(\Lambda)$ by using the argument as in \cite{LamCFT} (Proposition \ref{P:extwp}).

For the classification, we prove that a rootless even lattice $L$ can be constructed by Construction B if $V_L(1)\circ\tau$ is of untwisted type but is not a submodule of $V_L$ for some $\tau\in\Aut(V_L^{\hat{g}})$ (Proposition \ref{P:(I)}).
In this case, comparing the characters and fusion products of $V_L(1)$ and $V_L(1)\circ\tau$, we obtain an element $\lambda+L$ in the discriminant group $\mathcal{D}(L)$ of $L$ satisfying certain properties on $|(\lambda+L)(2)|$ and the order of $\lambda+L$ in $\mathcal{D}(L)$, which implies that $L$ can be constructed by Construction B from a code over $\Z_p$ (Theorem \ref{T:ChaB}).
In addition, we classify the rootless even lattices $L$ if $V_L(1)\circ\tau$ is of twisted type for some $\tau\in\Aut(V_L^{\hat{g}})$.
In this case, the quantum dimension of the irreducible $V_L^{\hat{g}}$-module of twisted type is $1$ (\cite{AbeLY,DJX}).
Hence there are a few possibilities for $\rank L$ and the discriminant group $\mathcal{D}(L)$, which is sufficient to classify all rootless even lattices $L$.
Furthermore, these lattices can be embedded into the Leech lattice as coinvariant lattices (Propositions \ref{P:357}, \ref{P:11A} and \ref{P:23A}).

The organization of the article is as follows: 
In Section 2, we recall some basic facts about lattices, lattice VOAs and the automorphism groups of lattice VOAs. We also recall the conjugates of modules. 
In Section 3, we review the irreducible modules over the orbifold $V_L^{\hat{g}}$ and study their conjugates by some automorphisms. 
In Section 4, we review the construction of certain even lattices using codes over $\Z_p$, which we call Construction A and Construction B.  We also discuss some characterizations of these lattices. 
In Section 5, we explain how to divide our arguments to two parts under the assumption that $g\in O(L)$ is fixed-point free of odd prime order.
In Section 6, we discuss the case where $V_L^{\hat{g}}$ has an extra automorphism $\tau$ such that $V_L(1)\circ\tau$ is of untwisted type.
We prove that such a $L$ is constructed by Construction B.
In addition, we construct an extra automorphism of $V_L^{\hat{g}}$ if $L$ is constructed by Construction B.
In Section 7, we discuss the case where $V_L^{\hat{g}}$ has an extra automorphism $\tau$ such that $V_L(1)\circ\tau$ is of twisted type.
We classify all possible rootless even lattices $L$; in fact, these are coinvariant lattices of the Leech lattice. 
In addition, we prove the existence of extra automorphisms of $V_L^{\hat{g}}$ for such lattices $L$ by using the Leech lattice VOA.

\section*{Acknowledgments}
The authors would like to thank the reviewers for giving them very useful suggestions and comments.

\medskip

\begin{center}
{\bf Notations}
\begin{small}
\begin{longtable}{ll} 
$(\cdot|\cdot)$& the positive-definite symmetric bilinear form of $\R^m$.\\
$\langle\cdot|\cdot\rangle$& the inner product on $\Z_p^t$.\\
$\mathcal{D}(L)$& the discriminant group $L^*/L$ of an even lattice $L$.\\
$\varepsilon$& the conformal weight \eqref{Eq:esp} of an irreducible $\hat{g}^s$-twisted $V_L$-module.\\
$\varepsilon(M)$& the conformal weight of a module $M$ over a VOA.\\
$g_\Delta$& the isometry in \eqref{Eq:g} of the root lattice of type $A_k$ associated with a base $\Delta$.\\
$g_{\Delta,e}$& the isometry in \eqref{Eq:gre} of $L_A(C)$ associated with a base $\Delta$ of $R$ and $e\in \mathcal{D}(R)$.\\
$L_A(C),L_B(C)$& the lattices in \eqref{Eq:ConstA} and \eqref{Eq:ConstB} associated with a subgroup $C$ of $\bigoplus_{i=1}^{t}\Z_{k_i}$.\\
$L^g$, $L_g$& $L^g=\{\alpha\in L\mid g\alpha=\alpha\}$ and $L_g=\{\alpha\in L\mid (\alpha| L^g)=0\}$ for $g\in O(L)$.\\
$L^*$& the dual lattice $\{ \alpha\in \Q\otimes_\Z L\mid (\alpha|  L) \subset \Z \}$ of a lattice $L$.\\
$\Lambda$& the Leech lattice.\\
$\lambda_1,\dots,\lambda_{k-1}$& fundamental weights of the root lattice of type $A_{k-1}$.\\
$\lambda_x$& $(\lambda_{x_1}^1,\dots,\lambda_{x_t}^t)\in R_1^*\perp\dots\perp R_t^*$ associated with $x=(x_i)\in\bigoplus_{i=1}^t\Z_{k_i}$ as in \eqref{Eq:lambdac}.\\
$M\circ\tau$& the $\tau$-conjugate of a $V$-module $M$ for $\tau\in\Aut(V)$ as in Definition \ref{Mconj}.\\
$O(L)$& the isometry group of a lattice $L$.\\
$\hat{g}$& a standard lift in $O(\hat{L})$ of an isometry $g$ of $L$.\\
$\rho_{\Delta}$& the Weyl vector associated with a base $\Delta$ of a root system.\\
$R$& the orthogonal sum of root lattices $R_1,\dots,R_t$ of type $A_{k_1-1},\dots,A_{k_t-1}$.\\
$\sigma_x$& the inner automorphism $\exp(-2\pi\sqrt{-1}x_{(0)})$ of a VOA $V$ associated with $x\in V_1$.\\
$S(2)$& $S(2)=\{\alpha\in S\mid (\alpha| \alpha)=2\}$ for a subset $S\subset \R^m$.\\
$V^\tau$& the fixed-point subVOA $\{v\in V\mid \tau v=v\}$ of $V$ for $\tau\in\Aut(V)$.\\
$\chi_\Delta$& the vector $(\rho_{\Delta_1}/k_1,\dots,\rho_{\Delta_t}/k_t)$ as in \eqref{Eq:chi}.\\
$\Z_k^\times$& the unit group of $\Z_k=\Z/k\Z$.\\
\end{longtable}
\end{small}
\end{center}

\addtocounter{table}{-1}

\section{Preliminaries}
\subsection{Lattices}
Let $L$ be a lattice with the positive-definite bilinear form $( \cdot | \cdot )$.
The \emph{dual lattice} $L^*$ of $L$ is defined to be the lattice
\[
L^*=\{ \alpha\in \Q\otimes_\Z L\mid (\alpha|  L) \subset \Z \}.
\]
A sublattice $K$ of $L$ is said to be \emph{full} if  $\rank K=\rank L$, namely, $\Q\otimes_\Z K\cong \Q\otimes_\Z L$. 
We denote the isometry group of $L$ by $O(L)$. 
Note that $O(L)$ is finite.

Assume that $L$ is \emph{even}, that is, $(\alpha|\alpha)\in 2\Z$ for all $\alpha\in L$.
Then $(\alpha|\beta)\in\Z$ for all $\alpha,\beta\in L$, and $L\subset L^*$.
The \emph{discriminant group} $\mathcal{D}(L)$ is defined to 
be the quotient group $L^*/L$, which is a 
finite  abelian group.  
The (square) \emph{norm} of a vector $\alpha\in \Q\otimes_\Z L$ is defined to be the value $(
\alpha| \alpha )\in\Q$.
For a subset $S\subset \Q\otimes_\Z L$, we denote by
\[
  S(2)=\{ \alpha\in S\mid (\alpha| \alpha) =2\},
\]
the set of all norm $2$ vectors in $S$.
We call a vector of norm $2$ a \emph{root}; in fact, for an even lattice $L$, $L(2)$ forms a simply laced root system.
An even lattice  $L$ is said to be \emph{rootless} if $L(2)=\emptyset$. 

Let $g\in O(L)$ of order $n$.
The fixed-point sublattice $L^g$ of $g$ and the \emph{coinvariant lattice} $L_g$ of $L$ associated with $g$ are defined to be 
\begin{equation}
L^g=\{\alpha\in L\mid g\alpha=\alpha\}\quad\text{and}\quad L_g=\{\alpha\in L\mid (\alpha| L^g)=0\},\label{Eq:L_g}
\end{equation}
respectively.
Clearly, $\rank L_g+\rank L^g=\rank L$ and the restriction of $g$ to $L_g$ is a fixed-point free isometry of order $n$; we often use the same symbol $g$ for the restriction.
 
We now give the following lemma, which will be used later.

\begin{lemma}\label{Lem:1-g+} Let $L$ be a lattice and let $g$ be a fixed-point free isometry of $L$.
Then  we have $((1-g)L)^*=(1-g)^{-1}L^*$.
\end{lemma}
\begin{proof}
Since $g$ is fixed-point free, $1$ is not an eigenvalue of $g$ and $(1-g)$ is non-singular.

It follows from $(\alpha|(1-g^{-1})\beta)=((1-g)\alpha|\beta)$ for $\alpha,\beta\in L$ that 
$((1-g^{-1})L)^*=(1-g)^{-1}L^*$.
Since $(1-g) L = -(1-g^{-1}) g L = (1-g^{-1}) L$, we have $((1-g)L)^*=(1-g)^{-1}L^*$.
\end{proof}

\subsection{VOAs, modules and automorphisms}
A \emph{vertex operator algebra} (VOA) $(V,Y,\vac,\omega)$ is a $\Z$-graded vector space $V=\bigoplus_{m\in\Z}V_m$ over the complex field $\C$ equipped with a linear map
$$Y(a,z)=\sum_{i\in\Z}a_{(i)}z^{-i-1}\in ({\rm End}\ V)[[z,z^{-1}]],\quad a\in V,$$
the \emph{vacuum vector} $\vac\in V_0$ and the \emph{conformal vector} $\omega\in V_2$
satisfying certain axioms (\cite{Bo,FLM}). 
The operators $L(m)=\omega_{(m+1)}$, $m\in \Z$, satisfy the Virasoro relation:
$$[L{(m)},L{(n)}]=(m-n)L{(m+n)}+\frac{1}{12}(m^3-m)\delta_{m+n,0}c\ {\rm id}_V,$$
where $c\in\C$ is called the \emph{central charge} of $V$.

Let $V$ be a VOA.
A $V$-\emph{module} $(M,Y_M)$ is a $\C$-graded vector space $M=\bigoplus_{m\in\C} M_{m}$ equipped with a linear map
$$Y_M(a,z)=\sum_{i\in\Z}a_{(i)}z^{-i-1}\in (\End\ M)[[z,z^{-1}]],\quad a\in V$$
satisfying a number of conditions (\cite{FHL,DLM2}).
We often denote it by $M$.
If $M$ is irreducible, then there exists $\varepsilon(M)\in\C$ such that $M=\bigoplus_{m\in\Z_{\geq 0}}M_{\varepsilon(M)+m}$ and $M_{\varepsilon(M)}\neq0$; the number $\varepsilon(M)$ is called the \emph{conformal weight} of $M$.
Let $\irr(V)$ denote the set of all isomorphism classes of irreducible $V$-modules.
We often identify an element in $\irr(V)$ with its representative.
Assume that the fusion product $\boxtimes$ is defined on $V$-modules (cf.\ \cite{HL}).
An irreducible $V$-module $M^1$ is called a \emph{simple current module} if for any irreducible $V$-module $M^2$, the fusion product $M^1\boxtimes M^2$ is also an irreducible $V$-module.

A linear automorphism $\tau$ of $V$ is called an (VOA) \emph{automorphism} of $V$ if $$ \tau\omega=\omega\quad {\rm and}\quad \tau Y(v,z)=Y(\tau v,z)\tau\quad \text{ for all } v\in V.$$
We denote the group of all automorphisms of $V$ by $\Aut(V)$. 
Note that $\Aut(V)$ preserves $V_n$ for every $n\in\Z$.
For an automorphism $\tau$ of $V$, let $V^\tau$ denote the orbifold of $V$ for $\langle\tau\rangle$;  $V^\tau=\{v\in V\mid \tau v=v\}$ is the fixed-point subVOA.

\begin{definition}\label{Mconj}
Let $V$ be a VOA and let $M=(M,Y_M)$ be a $V$-module.
For $\tau\in\Aut(V)$, the \emph{$\tau$-conjugate} of $M$ is the $V$-module $(M\circ \tau, Y_{M\circ \tau} )$ defined as follows:
\begin{equation}
\begin{split}
& M\circ \tau =M \quad \text{ as a vector space;}\\
& Y_{M\circ \tau} (a, z) = Y_M(\tau a, z)\quad \text{ for any } a\in V.
\end{split}\label{Eq:conjact}
\end{equation}
\end{definition}

If $V$-modules $M^1$ and $M^2$ are isomorphic, then $M^1\circ \tau$ and $M^2\circ \tau$ are also isomorphic.
Hence, the $\tau$-conjugate defines an $\Aut(V)$-action on the isomorphism classes of $V$-modules.  
The following lemma is immediate:
\begin{lemma}\label{Lem:conj} Let $M,M^1,M^2$ be $V$-modules and let $\tau\in\Aut(V)$.
Assume that the fusion product is defined on $V$-modules.
\begin{enumerate}[{\rm (1)}]
\item If $M$ is irreducible, then so is $M\circ \tau$.
In addition, $\varepsilon(M)=\varepsilon(M\circ \tau)$ and $\dim M_{\varepsilon(M)+i}=\dim(M\circ\tau)_{\varepsilon(M)+i}$ for all $i\in\Z_{\ge0}$;
\item $(M^1\circ \tau)\boxtimes (M^2\circ \tau)\cong (M^1\boxtimes M^2)\circ \tau$.
In particular, if $M$ is a simple current module, then so is $M\circ\tau$.
\end{enumerate}
\end{lemma}

An irreducible (untwisted or twisted) module $M$ of $V_L$ is said to be \emph{$\hat{g}$-stable} if 
$M\circ \hat{g}\cong M$. In this case, $\hat{g}$ acts on $M$ and we denote  the eigenspaces of $\hat{g} $ on $M$ by 
\begin{equation}\label{jeigen}
M(j)= \{ x\in M\mid  \hat{g} x= e^{2\pi \sqrt{-1}j / n} x\}, \quad  0 \leq j\leq n-1,\ n=|\hat{g}|.    
\end{equation}

\subsection{Lattice VOAs and their automorphism groups}\label{sec:2.3}
Next we recall a few facts about lattice VOAs and their automorphism groups from \cite{FLM,DN, LY2}.

Let $L$ be an even lattice and let $(\cdot |\cdot )$ be the positive-definite symmetric bilinear form on $\R\otimes_\Z L$.
Let $M(1)$ be the Heisenberg VOA associated with $\mathfrak{h}=\C\otimes_\Z L$ and the form $(\cdot|\cdot)$ extended $\C$-bilinearly. Let $\C\{L\}=\bigoplus_{\alpha\in L}\C e^\alpha$ be the twisted group algebra such that $e^\alpha e^\beta=(-1)^{(\alpha|\beta)}e^{\beta}e^{\alpha}$, for  $\alpha,\beta\in L$.
The lattice VOA $V_L$ associated with $L$ is defined to be $M(1) \otimes \C\{L\}$ (\cite{FLM}). 

Let $\hat{L}$ be a central extension of $L$ associated with the commutator map $(\cdot|\cdot)\pmod2$ on $L$.
Let $\Aut(\hat{L})$ be the automorphism group of $\hat{L}$.
For $\varphi\in \Aut (\hat{L})$, we define the element $\bar{\varphi}\in\Aut(L)$ by $\varphi(e^\alpha)\in\{\pm e^{\bar\varphi(\alpha)}\}$, $\alpha\in L$.
Set $O(\hat{L})=\{\varphi\in\Aut(\hat L)\mid \bar\varphi\in O(L)\}.$
For $x\in\h$, we set 
\begin{equation}
\sigma_x=\exp(-2\pi\sqrt{-1}x_{(0)})\in\Aut(V_L).\label{Eq:sigma}
\end{equation}
Then $\{\varphi\in O(\hat{L})\mid \bar{\varphi}=id\}=\{\sigma_x\mid x\in ((1/2)L^*)/L^*\}.$
By \cite[Proposition 5.4.1]{FLM}, we have an exact sequence
\[
  1 \to \{\sigma_x\mid x\in ((1/2)L^*)/L^*\} \to  O(\hat{L})\stackrel{-}\rightarrow  O(L) \to  1.
\]
Note that $\aut (V_L) = N(V_L)\,O(\hat{L})$ (\cite{DN}), where $N(V_L)=\l\la \exp(a_{(0)}) \mid a\in (V_L)_1 \r\ra$.

An element $\phi\in O(\hat{L})$ is called a \emph{lift} of $g\in O(L)$ if $\bar{\phi}=g$ and it is called a \emph{standard lift} of $g\in O(L)$ if $\phi$ is a lift of $g$ and $\phi(e^\alpha)=e^\alpha$ for $\alpha\in L^g$.
The following lemma can be found in \cite{EMS,LS20}.
\begin{lemma}\label{L:standardlift}
Let $g\in O(L)$ and let $\hat{g}\in O(\hat{L})$ be a standard lift of $g$.
\begin{enumerate}[{\rm (1)}]
\item Any standard lift of $g$ is conjugate to $\hat{g}$ by an element in $\Aut(V_L)$.
\item If $g$ has odd order, then $|\hat{g}|=|g|$.
\end{enumerate}
\end{lemma}

If $g\in O(L)$ is fixed-point free, then $L^g=\{0\}$, which implies that any lift of $g$ is a standard lift.
Note that for $s\in\Z$ with $\gcd(s,|g|)=1$, $g^s$ is also fixed-point free.
\begin{lemma}\label{L:fixedfree}
For any lift $\hat{g}\in O(\hat{L})$ of a fixed-point free isometry $g$,  $|\hat{g}|=|g|$.
\end{lemma}

\begin{proof}
	Let $\eta: L\to \{\pm 1\}$ such that $\hat{g}(e^\alpha) =\eta(\alpha)e^{g\alpha}$. 
	Let $n$ be the order of $g$.
	By the proof of \cite[Proposition 7.2]{EMS}, 
	\[
	\eta(\alpha+ g\alpha +\dots+g^{n-1}\alpha)= (-1)^{(\alpha|g\alpha +\dots+g^{n-1}\alpha )}
	\eta(\alpha) \eta(g\alpha)\cdots\eta(g^{n-1}\alpha). 
	\]
Since $g$ is fixed-point free of order $n$, $\alpha+ g\alpha +\dots+g^{n-1}\alpha=0$ and 	$
(\alpha|g\alpha +\dots+g^{n-1}\alpha )=-(\alpha|\alpha)\in 2\Z$; thus 
\[
\hat{g}^n(e^\alpha) =\eta(\alpha) \eta(g\alpha)\cdots\eta(g^{n-1}\alpha) e^\alpha =e^\alpha.
\]
\end{proof}

Let $g\in O(L)$ be a fixed-point free isometry of prime order $p$ and let $\hat{g}\in O(\hat{L})$ be a lift of $g$.
By the proof of \cite[Theorem 5.15]{LY2}, along with  Lemma \ref{Lem:1-g+}, we obtain
\begin{equation}
\{\sigma_y\mid y\in \h/L^*\}\cap N_{\Aut(V_L)}(\langle\hat{g}\rangle)=\{\sigma_y\mid y\in ((1-g)^{-1}L^*)/L^*\},
\label{Eq:hom}
\end{equation}
where $N_{\Aut(V_L)}(\langle\hat{g}\rangle)$ is the normalizer of $\langle\hat{g}\rangle$ in $\Aut(V_L)$.

Clearly, $N_{\aut(V_L)}(\langle \hat{g}\rangle)$ preserves $V_L^{\hat{g}}$.
Since $V_L$ is a simple current extension of $V_L^{\hat{g}}$ graded by $\Z/p\Z$, $N_{\aut(V_L)}(\langle \hat{g}\rangle)/\langle \hat{g}\rangle$ acts faithfully on $V_L^{\hat{g}}$, that is, $$N_{\aut(V_L)}(\langle \hat{g}\rangle)/ \langle \hat{g}\rangle \subset \aut(V_L^{\hat{g}}).$$

\begin{definition}\label{Def:extra}
An automorphism $\tau$ of $V_L^{\hat{g}}$ is said to be \emph{extra} if $\tau\notin N_{\aut(V_L)}(\langle \hat{g}\rangle)/ \langle \hat{g}\rangle$.
\end{definition}

A characterization of extra automorphisms has been obtained in \cite{Sh04}:

\begin{proposition}{\rm (cf.\ \cite[Theorem 3.3]{Sh04}\label{P:stabVL1})}
An automorphism $\tau$ of $V_L^{\hat{g}}$ is extra
if and only if $V_L\circ \tau\not\cong V_L$ as $V_L^{\hat{g}}$-modules.
\end{proposition}

In order to determine $\aut(V_L^{\hat{g}})$, the key question is to determine if $N_{\aut(V_L)}(\langle \hat{g}\rangle)/ \langle \hat{g}\rangle$ is equal to $\aut(V_L^{\hat{g}})$ or not, equivalently, if $V_L^{\hat{g}}$ has extra automorphisms or not. 
The main aim of this article is to discuss necessary and sufficient conditions for which the VOA $V_L^{\hat{g}}$ has extra automorphisms. 

\section{Irreducible $V_L^{\hat{g}}$-modules and their conjugates}
Let $L$ be a positive-definite even lattice and let $g$ be a fixed-point free isometry of $L$ of prime order $p$.
Let $\hat{g}\in O(\hat{L})$ be a lift of $g$.
Then the order $|\hat{g}|$ of $\hat{g}$ is also $p$.
It follows from \cite{CM,Mc22,Mi} that $V_L^{\hat{g}}$ is rational, $C_2$-cofinite, self-dual and of CFT-type.
Hence the fusion product is defined on $V_L^{\hat{g}}$-modules.

It was proved in \cite{DRX} that any irreducible $V_L^{\hat{g}}$-module appears as a submodule of an irreducible $\hat{g}^s$-twisted $V_L$-module for some $0\le s\le p-1$.
Note that a $\hat{g}^0$-twisted $V_L$-module is a $V_L$-module.

\begin{definition}\label{D:untw} An irreducible $V_L^{\hat{g}}$-module is said to be \emph{of $\hat{g}^s$-type} if it is a $V_L^{\hat{g}}$-submodule of an irreducible $\hat{g}^s$-twisted $V_L$-module.
In addition, it is said to be \emph{of untwisted type} (resp. \emph{of twisted type}) if it is of $\hat{g}^0$-type (resp. of $\hat{g}^s$-type for some $1\le s\le p-1$).
\end{definition}

In this section, we review the constructions of irreducible $V_L^{\hat{g}}$-modules and describe the conjugates of irreducible $V_L^{\hat{g}}$-modules by some elements in 
$\{\sigma_y\mid y\in \h/L^*\}\cap N_{\Aut(V_L)}(\langle\hat{g}\rangle)$. 

\subsection{Irreducible $V_L^{\hat{g}}$-modules of untwisted type}\label{S:un}
We first discuss the irreducible $V_L^{\hat{g}}$-modules of untwisted type.

Let $\lambda+L\in \mathcal{D}(L)$ and $V_{\lambda+L}=M(1)\otimes\Span_\C\{e^\alpha\mid \alpha\in\lambda+L\}$.
Then $V_{\lambda+L}$ has an irreducible $V_L$-module structure (\cite{FLM}).

Assume that $(1-g)\lambda\notin L$, that is, $g(\lambda+L)\neq\lambda+L$.
Then $V_{\lambda+L}$ is also irreducible as a $V_L^{\hat{g}}$-module.
Note that $V_{\lambda+L}$ is not a simple current module over $V_L^{\hat{g}}$ (\cite{DJX}).

Assume that $(1-g)\lambda\in L$, that is, $g(\lambda+L)=\lambda+L$.
Then $V_{\lambda+L}$ is $\hat{g}$-stable; we fix a $\hat{g}$-module isomorphism $\hat{g}_{\lambda+L}$ of $V_{\lambda+L}$ of order $p$.
For $0\le i\le p-1$, we denote 
\[V_{\lambda+L}(i)=\{v\in V_{\lambda+L}\mid \hat{g}_{\lambda+L}(v)=\exp(2\pi\sqrt{-1}i/p)v\},
\] 
which is an irreducible $V_L^{\hat{g}}$-module. Note also that $V_{\lambda+L}(i)$ is a simple current module over $V_L^{\hat{g}}$ \cite{DJX}. 
Recall from \cite{dong} that the fusion product $V_{\lambda+L} \boxtimes V_{\lambda'+L} = V_{\lambda+\lambda'+L}$ holds for $\lambda+L,\lambda'+L\in\mathcal{D}(L)$.
Hence, if $(1-g)\lambda,(1-g)\lambda'\in L$, then for $0\le i,i'\le p-1$, there exists $0\le i''\le p-1$ such that 
\begin{equation}
V_{\lambda+L}(i)\boxtimes V_{\lambda'+L}(i')\cong V_{\lambda+\lambda'+L}(i'')\label{Eq:fusionun}.
\end{equation}

In addition, for $1\le i\le p-1$, $0\le j\le p-1$, and $\mu+L\in\mathcal{D}(L)\setminus \{L\}$ with $(1-g)\mu\in L$, the conformal weights are given by 
\begin{equation}
\varepsilon(V_L(0))=0,\quad \varepsilon(V_L(i))=1,\quad \varepsilon(V_{\mu+L}(j))=\frac{1}{2}\min\{(v| v)\mid v\in \mu+L\}.\label{Eq:cwun}
\end{equation}

For $\alpha\in L^*$, the inner automorphism $\sigma_{(1-g)^{-1}\alpha}$ belongs to $\Aut(V_L^{\hat{g}})$ by \eqref{Eq:hom}.

\begin{lemma}\label{Lem:conjhom} Let $\alpha,\lambda \in L^*$ with $g(\lambda+L)=\lambda+L$ and let $0\le i,j\le p-1$.
If $(\alpha|\lambda)\in j/p+\Z$, then as $V_L^{\hat{g}}$-modules, $$V_{\lambda+L}(i)\circ \sigma_{(1-g)^{-1}\alpha}\cong V_{\lambda+L}(i-j).$$ 
\end{lemma}
\begin{proof} Set $y=(1-g)^{-1}\alpha$.
Since $\sigma_y$ is an inner automorphism of $V_L$, it can be extended to a $V_L$-module isomorphism from $V_{\lambda+L}$ to $V_{\lambda+L}\circ \sigma_y\cong V_{\lambda+L}$.
Then $\hat{g}_{\lambda+L}$ acts on $V_{\lambda+L}\circ \sigma_y$ as $\sigma_y\hat{g}_{\lambda+L}\sigma_y^{-1}=\sigma_{(1-g)y}\hat{g}_{\lambda+L}=\sigma_\alpha\hat{g}_{\lambda+L}$.
By the assumption, $\sigma_{\alpha}$ acts on $V_{\lambda+L}$ as the scalar multiplication $\exp(2\pi\sqrt{-1}(-j)/p)$.
Hence $\hat{g}_{\lambda+L}$ acts on $V_{\lambda+L}(i)\circ\sigma_y$ as the scalar multiplication $\exp(2\pi\sqrt{-1}(i-j)/p)$, which proves this lemma.
\end{proof}

\subsection{Irreducible $V_L^{\hat{g}}$-modules of twisted type}\label{S:tw}

Next, we discuss the irreducible $V_L^{\hat{g}}$-modules of twisted type.
We use the descriptions in \cite[Sections 3.2 and 3.3]{AbeLY} (see also \cite{Le, DL, BK04}).
Let $1\le s\le p-1$.
The irreducible $\hat{g}^s$-twisted module $V_{L}^{T}[\hat{g}^s]$ is given by 
\begin{equation*}
V_L^{T}[\hat{g}^s]= M(1)[{g}^s]\otimes T, 
\end{equation*}
where $M(1)[{g}^s]$ is the ``${g}^s$-twisted" free bosonic space and $T$ is an irreducible module for a certain ``$\hat{g}^s$-twisted" central extension of $L$.
Its conformal weight is given in \cite[(6.28)]{DL} as
\begin{equation}
\varepsilon=\varepsilon(V_L^{T}[\hat{g}^s])=\frac{1}{4p^2}\frac{m}{p-1}\sum_{k=1}^{p-1}k(p-k)=\frac{m(p+1)}{24p},\label{Eq:esp}
\end{equation}
where $m$ is the rank of $L$.   
By the explicit construction of $V_L^{T}[\hat{g}^s]$, we obtain the following:
\begin{lemma}\label{L:cwtw} The conformal weight of any irreducible $V_L^{\hat{g}}$-module of twisted type belongs to
$\{\varepsilon+i/p\mid 0\le i\le p-1\}$.
\end{lemma}

By \cite[Proposition 6.2]{Le}, we have  
\begin{equation}\label{dimT}
(\dim T)^2= [L:{R}_L^{{g}^s}],
\end{equation}
where ${R}_L^{{g}^s}=((1-g^s)L^* )\cap L$.   It was shown in \cite{DLM2} that all irreducible $\hat{g}^i$-twisted modules are $\hat{g}$-stable  and  we denote 
\[
V_{L}^T[\hat{g}^s] (j) =  \{ x\in V_{L}^T[\hat{g}^s]\mid  \hat{g} x= e^{2\pi \sqrt{-1}j / p} x\}, \quad  0 \leq j\leq n-1. 
\]

\subsection{Irreducible $V_L^{\hat{g}}$-modules under $(1-g)L^*\subset L$}\label{S:1-g}

In this subsection, we assume that 
\begin{equation}
(1-g)L^*\subset L,\quad \text{equivalently},\quad g(\lambda+L)=\lambda+L\quad \text{ for all } \lambda+L\in\mathcal{D}(L).\label{Eq:as1}
\end{equation}
Note that the above condition holds if $V_L(1)$ is conjugate to some irreducible module of twisted type; the results in this subsection will be used in Section \ref{S:2} for the classification  of certain lattices. 

Under our assumption, $V_L$ has exactly  $|\mathcal{D}(L)|$ irreducible $\hat{g}^s$-twisted modules for any $0 \leq s \leq p-1$ \cite{DLM2}. Moreover, for any irreducible $\hat{g}^s$-twisted modules  $V_L^T[\hat{g}^s]$,   
the module $T$ in Section \ref{S:tw}  
is determined by a linear character of $(1-g^s)L^*/(1-g^s)L$ and  can be  labeled by $(1-g^s)L^*/(1-g^s)L\cong\mathcal{D}(L)$ \cite[Proposition 6.2]{Le}; 
we adopt the following notation as in \cite{Lam19}:
\begin{equation} \label{VlL}
V_{\lambda+L}[\hat{g}^s]=M(1)[g^s]\otimes T_{\lambda+L},\quad \lambda+L\in \mathcal{D}(L).
\end{equation}
It follows from Lemma \ref{Lem:1-g+} (2) that the set of linear characters of $(1-g^s)L^*/(1-g^s)L$ is $$\{\exp(-2\pi\sqrt{-1}(\cdot\mid(1-g^s)^{-1}\lambda))\mid \lambda+L\in \mathcal{D}(L)\}$$
and we choose the labeling so that $T_{\lambda+L}$ is associated with $\exp(-2\pi\sqrt{-1}(\cdot\mid(1-g^s)^{-1}\lambda))$.
Hence we obtain the following:

\begin{lemma}\label{Lem:conjun0} Let $\alpha\in L^*$ and $1\le s\le p-1$.
Then, as $\hat{g}^s$-twisted $V_L$-modules, $V_{L}[\hat{g}^s]\circ \sigma_{(1-g^s)^{-1}\alpha}\cong V_{\alpha+L}[\hat{g}^s].$
\end{lemma}

It also follows from \cite[Theorem 4.11]{Lam19} that all irreducible $V_L^{\hat{g}}$-module are simple current modules. Thus, 
\begin{equation}
\irr(V_L^{\hat{g}})=\{V_{\lambda+L}[\hat{g}^s](i)\mid \lambda+L\in\mathcal{D}(L),\quad 0\le i,s\le p-1\},\label{Eq:irr}
\end{equation}
is the set of all inequivalent irreducible modules for $V_L^{\hat{g}}$, 
where we identify $V_{\lambda+L}(i)$ with $V_{\lambda+L}[\hat{g}^0](i)$. 

We further assume that there exist $\tau\in\Aut(V_L^{\hat{g}})$, $\mu+L\in\mathcal{D}(L)$, $0\le q\le p-1$ and $1\le s\le p-1$ such that 
\begin{equation}
V_L(1)\circ\tau\cong V_{\mu+L}[\hat{g}^s](q)\label{Eq:assumptw}.
\end{equation}
In other words, we assume that $V_L(1)$ is conjugate to an irreducible module of twisted type by the action of $\Aut(V_L^{\hat{g}})$. 

By Lemma \ref{Lem:conj} (1) and \eqref{Eq:cwun}, $\varepsilon(V_{\mu+L}[\hat{g}^s](q))=\varepsilon(V_L(1))=1$.
By Lemma \ref{L:cwtw}, we have 
\begin{equation*}
\varepsilon=\varepsilon(V_{\mu+L}[\hat{g}^s]) \in\frac{1}{p}\Z.
\end{equation*}
We choose the labeling of irreducible $V_L^{\hat{g}}$-submodules of $V_{\lambda+L}[\hat{g}^s]$ as in \cite{Lam19} so that
$$\varepsilon(V_{\lambda+L}[\hat{g}^s](j))\equiv\frac{sj}{p}+\frac{(\lambda|\lambda)}{2}\pmod\Z.$$
The following lemma is immediate from Lemma \ref{Lem:conjun0}:

\begin{lemma}\label{Lem:conjun} Let $\alpha\in L^*$ and $1\le s\le p-1$.
For $0\le j\le p-1$, $$V_L[\hat{g}^s](j)\circ\sigma_{(1-g^s)^{-1}\alpha}\cong V_{\alpha+L}[\hat{g}^s](j'),$$where $j'$ is determined by ${sj}\equiv p(\alpha|\alpha)/2+sj'\pmod{p}$.
In particular, all irreducible $V_L^{\hat{g}}$-modules of $\hat{g}^s$-type with the same conformal weight are conjugate under $\Aut(V_L^{\hat{g}})$.
\end{lemma}
In addition, by \cite[Theorem 5.3]{Lam19}, $\irr(V_L^{\hat{g}})\cong \mathcal{D}(L)\times\Z_p^2$ as abelian groups under the fusion  product; the explicit multiplications are given as follows:
\begin{equation}
V_{\lambda_1+L}[\hat{g}^{s_1}](i_1)\boxtimes V_{\lambda_2+L}[\hat{g}^{s_2}](i_2)=V_{\lambda_1+\lambda_2+L}[\hat{g}^{s_1+s_2}](i_1+i_2).\label{Eq:fusionprod}
\end{equation}
By Lemma \ref{Lem:conj} (2), the assumption \eqref{Eq:assumptw} and the fusion product \eqref{Eq:fusionprod}, we obtain 
\begin{equation}
\{V_L(j)\circ\tau\mid 1\le j\le p-1\}=\{V_{j\mu+L}[\hat{g}^{js}](jq)\mid 1\le j\le p-1\}.\label{Eq:twistj}
\end{equation}

In the remaining of this subsection, we always assume that \eqref{Eq:as1} and \eqref{Eq:assumptw} hold.

\begin{proposition}\label{P:twist} For any $\lambda+L\in \mathcal{D}(L)$, $V_{\lambda+L}(0)\circ\tau$ or $V_{\lambda+L}(1)\circ\tau$ is of twisted type.
\end{proposition}
\begin{proof} By \eqref{Eq:fusionprod}, we have $V_L(1)\boxtimes V_{\lambda+L}(0)=V_{\lambda+L}(1)$.
Then we obtain $(V_L(1)\circ \tau)\boxtimes(V_{\lambda+L}(0)\circ \tau)=(V_{\lambda+L}(1)\circ \tau)$ by Lemma \ref{Lem:conj} (2).
If both $V_{\lambda+L}(0)\circ \tau$ and $V_{\lambda+L}(1)\circ \tau$ are of untwisted type, then by \eqref{Eq:fusionprod}, so is $V_L(1)\circ \tau$, which contradicts the assumption \eqref{Eq:assumptw}.
\end{proof}

\begin{corollary}\label{C:confwt} 
Let $1\le s\le p-1$.
Let $\varepsilon$ be the rational number given in \eqref{Eq:esp}.
Then $\{\varepsilon(M)\mid M\in\irr(V_L^{\hat{g}})\}=\{0,\varepsilon+i/p\mid 0\le i\le p-1\}$.
In particular, $\{\varepsilon(M)\mid M\in\irr(V_L^{\hat{g}})\}\cap\Z=\{0,1\}$.
\end{corollary}
\begin{proof} Set $J=\{\varepsilon+i/p\mid 0\le i\le p-1\}$.
Let $\irr(V_L^{\hat{g}})^{tw}$ be the subset of $\irr(V_L^{\hat{g}})$ consisting of isomorphism classes of irreducible $V_L^{\hat{g}}$-modules of twisted type.
By Lemma \ref{L:cwtw} and \eqref{Eq:irr}, $\{\varepsilon(M)\mid M\in\irr(V_L^{\hat{g}})^{tw}\}=J$.
It follows from Proposition \ref{P:twist}, \eqref{Eq:cwun} and \eqref{Eq:assumptw} that $\varepsilon(V_L(1))=1\in J$ and $\varepsilon(V_{\lambda+L}(j))\in J$ for any $\lambda+L\in\mathcal{D}(L)\setminus\{L\}$ and $0\le j\le p-1$.
Hence by \eqref{Eq:cwun}, 
$\{\varepsilon(M)\mid M\in\irr(V_L^{\hat{g}})\}=J\cup\{0\}$.
\end{proof}

The following lemma is immediate from \eqref{Eq:cwun} and Corollary \ref{C:confwt}.

\begin{lemma}\label{L:norm2}
Let $\lambda+L\in \mathcal{D}(L)\setminus\{L\}$ such that $(\lambda| \lambda)\in 2\Z$. Then the minimum norm of $\lambda+L$ is $2$.
\end{lemma}

\begin{proposition} \label{LC}
Let $\lambda+L\in \mathcal{D}(L)\setminus\{L\}$ such that $(\lambda| \lambda)\in 2\Z$.
Then there exist $i\in\{0, 1\}$ and 
$1\le j\le p-1$ such that $V_L(j)$ and $V_{\lambda+L}(i)$ are conjugate by an element of $\Aut(V_L^{\hat{g}})$.
\end{proposition}
\begin{proof} 
By Lemma \ref{L:norm2}, the minimum norm of $\lambda+L$ is $2$.
By Proposition \ref{P:twist}, $V_{\lambda+L}(i)\circ \tau$ is of $\hat{g}^\ell$-type for some $i\in\{0,1\}$ and $1\le \ell\le p-1$, and  its conformal weight is $1$.
By \eqref{Eq:twistj}, $V_L(j)\circ\tau$ is of $\hat{g}^\ell$-type for some $1\le j\le p-1$ and its conformal weight is $1$.
Then, the assertion follows from Lemma \ref{Lem:conjun}.
\end{proof}

\section{Lattices associated with subgroups of $\bigoplus_{i=1}^t\Z_{k_i}$}

In this section, we discuss constructions of lattices associated with subgroups of $\bigoplus_{i=1}^t\Z_{k_i}$, which we call Construction A and Construction B.
We also discuss their properties and give some characterizations.
In Sections \ref{Sec:root} and \ref{Sec:ConstAB}, $k_i$ is not necessary a prime.
However, in Sections \ref{Sec:Zp}, \ref{Sec:ChaAB} and \ref{S:example}, $k_1=\dots=k_t$ is a prime. 

\subsection{Root lattice of type $A_{k-1}$}\label{Sec:root}
Let $k\in\Z_{\ge2}$.
In this subsection, we review some basic properties of the root lattice $A$ of type $A_{k-1}$ (cf. \cite{CS}).
The following is a standard model for $A$:
\[
A=\left \{(a_1, a_2, \dots, a_{k}) \in \Z^{k} \, \left |\,  \sum_{i=1}^{k} a_i =0\right. 
\right \}.
\]
Let $\{v_1=(1,0,\dots,0), \dots, v_{k} =(0, 0, \dots, 1)\}$ be the standard basis of $\Z^{k}$.
Then $A(2)=\{\pm (v_i-v_j)\mid 1\leq i<j \leq {k}\}$ and it forms a root system of type $A_{k-1}$.

Set $\alpha_i= v_{i}-v_{i+1}$ for $1\le i\le k-1$. Then $\Delta=\{\alpha_i\mid 1\le i\le k-1\}$ is a base of the root system of type $A_{k-1}$ and has the following Dynkin diagram:
\[
\begin{array}{l}
        \circ
  \hspace{-5pt}-\hspace{-7pt}-\hspace{-5pt}-\hspace{-5pt}-\hspace{-5pt}
        \circ
  \hspace{-5pt}-\hspace{-5.5pt}-\hspace{-5pt}- \hspace{5pt}
        \cdots
  \hspace{5pt}-\hspace{-5pt}- \hspace{-6pt}-\hspace{-6pt}-\hspace{-5pt}
       \circ
  \hspace{-5pt}-\hspace{-5pt}-\hspace{-6pt}-\hspace{-6pt}- \hspace{-5pt}
      \circ
  \vspace{-10pt} \\
  \alpha_1\hspace{18pt} \alpha_2 \hspace{61pt} \alpha_{k-2}
  \hspace{16pt} \alpha_{k-1}\\
\end{array}
\]

Recall that $\mathcal{D}(A) \cong \Z_{k}$. Set $\lambda_{0}=0$ and
\[
\lambda_j = \frac{1}{k} \left
((k-j)\sum_{i=1}^j v_i- j \sum_{i=j+1}^{k} v_i\right)  \text{ for } j = 1, \dots ,k-1.
\]
Then $\lambda_1+A$ is a generator of $\mathcal{D}(A)$ and $\lambda_j\in j\lambda_1+A$.
In addition,
\begin{equation}
(\lambda_i|\lambda_j)=\frac{i(k-j)}{k}\quad {\rm if}\quad i\le j,\label{Eq:gammaj}
\end{equation} 
where $(\cdot|\cdot)$ is the standard inner product of $\R^k$.
Hence for $1\le i,j\le k-1$
\begin{equation}
(\lambda_i|\lambda_j)\equiv -\frac{ij}{k}\pmod{\Z}.\label{Eq:innerij}
\end{equation}

Let $\rho_{\Delta}$ be the Weyl vector with respect to $\Delta$, i.e., 
\begin{equation}
\rho_{\Delta}=\frac{1}2(k-1, k-3, \dots, -k+3, -k+1)=\sum_{j=1}^{k-1}\lambda_j\in\ A^*.\label{Eq:Weylvec}
\end{equation}
Note that for $1\le i\le k-1$, \begin{equation}
(\rho_{\Delta}| \alpha_i)=1\quad \text{and}\quad (\rho_{\Delta}| \lambda_i)=\frac{(k-i)i}{2}.\label{Eq:inner2}
\end{equation}
Hence if $k$ is odd (resp. even), then $\rho_\Delta\in A$ (resp. $\rho_\Delta\in (1/2)A\setminus A$).

\begin{lemma}\label{L:modp} Let $\beta$ be a root of $A$.
Then $(\beta| \rho_\Delta)\not\equiv0\pmod{k}$.
\end{lemma}
\begin{proof} 
If $\beta$ is positive (resp. negative), then $\beta=\sum_{i=1}^{k-1} c_i\alpha_i$ (resp. $\beta=-\sum_{i=1}^{k-1}c_i\alpha_i$) and $c_i\in\{0,1\}$.
Hence $1\le |(\rho_\Delta| \beta)|=\sum_{i=1}^{k-1}c_i\le k-1$.
\end{proof}

Set $\alpha_0=-\sum_{i=1}^{k-1}\alpha_i$, the negated highest root.
Denote $\tilde{\Delta}=\Delta\cup\{\alpha_0\}.$
Let $g_{\Delta}$ be the following fixed-point free isometry of $A$ of order $k$:  
\begin{equation}
g_{\Delta}(x_1,x_2,\dots,x_{k})=(x_{k},x_1,x_2,\dots,x_{k-1}).\label{Eq:g}
\end{equation}
Then $g_\Delta(\alpha_i)=\alpha_{i+1}$ if $1\le i\le k-2$ and $g_\Delta(\alpha_{k-1})=\alpha_0$.
In particular, $g_\Delta$ acts on $\tilde{\Delta}$ as a cyclic permutation of order $k$.
It is known that $g_\Delta$ is a Coxeter element of the Weyl group, the subgroup of $O(A)$ generated by all reflections associated with roots.
It is easy to see that for $1\le j\le k-1$
\begin{equation}
g_\Delta(\lambda_j)=\lambda_j-\sum_{i=1}^j\alpha_i,\quad g_{\Delta}(\rho_{\Delta})=\rho_{\Delta}-k\lambda_{1}.
\label{Eq:grho}
\end{equation}

The following lemma is immediate:

\begin{lemma}\label{L:fpf} For $1\le j\le k-1$, $g_\Delta^j$ is fixed-point free if and only if $\gcd(j,k)=1$.
\end{lemma}

We now give some lemmas, which will be used later.

\begin{lemma}\label{Lem:Ap1} Let $t$ be a positive integer.
Assume that $k\ge3$.
Let $X$ be a set of roots of the root system of type $A_{k-1}^t$.
Assume that $(\alpha|\beta)\in\{0,-1\}$ for all distinct $\alpha,\beta\in X$.
Then $|X|\le tk$.
In addition, if $|X|=tk$, then $X$ is the union of a base and the negated highest roots.
\end{lemma}
\begin{proof} Let $\Gamma$ be a graph on $X$ such that $\{\alpha,\beta\}$ is an edge if $(\alpha|\beta)=-1$.
Then $X$ is a subgraph of the union of $t$ copies of the extended Dynkin diagram of type $A_{k-1}$; otherwise, ${\rm Span}_\Z X$ has a negative norm vector.
Hence $|X|\le tk$. 
If $|X|=tk$, then $\Gamma$ is the union of $t$ copies of the extended Dynkin diagram of type $A_{k-1}$, which proves the latter assertion.
\end{proof}

\begin{lemma}\label{L:dih} Assume that $k$ is an odd prime.
Let $h\in O(A)$ be a fixed-point free isometry of order $k$ such that $h$ preserves $\tilde{\Delta}$.
Then $h\in\langle g_\Delta\rangle$.
\end{lemma}
\begin{proof} Notice that the stabilizer of $\tilde\Delta$ in $O(A)$ is the stabilizer of the extended Dynkin diagram of type $A_{k-1}$, and it is the dihedral group of order $2k$ generated by $g_\Delta$ and the diagram automorphism of order $2$.
Since $k$ is an odd prime, $h\in\langle g_\Delta\rangle$.
\end{proof}

\subsection{Constructions of lattices associated with subgroups of $\bigoplus_{i=1}^t\Z_{k_i}$}\label{Sec:ConstAB}
Let $t$ be a positive integer and let $k_i\in\Z_{\ge2}$ for $1\le i\le t$.
Let $R_i$ $(1\le i\le t)$ be a copy of the root lattice of type $A_{k_i-1}$.
Let $R$ be the orthogonal sum of $R_1,\dots,R_t$; $$R=R_1\perp R_2\perp\dots\perp R_t.$$
Then $\mathcal{D}(R_i)\cong\Z_{k_i}$ and $\mathcal{D}(R)\cong\bigoplus_{i=1}^t\Z_{k_i}$.
Let $$\nu:R^*\to R^*/R=\mathcal{D}(R)\cong\bigoplus_{i=1}^t\Z_{k_i}$$ be the canonical surjective map.
For a subgroup $C$ of $\bigoplus_{i=1}^t\Z_{k_i}$, let $L_A(C)$ denote the lattice defined by 
\begin{equation}
L_A(C)=\nu^{-1}(C)=\{\alpha\in R^*\mid \nu(\alpha)\in C\};\label{Eq:ConstA}
\end{equation}
we call $L_A(C)$ the lattice constructed by \emph{Construction A} from $C$.
Since $L_A(C)$ contains $R$ as a full sublattice, the rank of $L_A(C)$ is $\sum_{i=1}^t(k_i-1)$.
Note that $L_A(\{\allzero\})=R$, where $\allzero$ is the identity element of $\bigoplus_{i=1}^t\Z_{k_i}$.

We now fix a base $\Delta_i$ of the root system $R_i(2)$ of type $A_{k_i-1}$.
Then $\Delta=\bigcup_{i=1}^t\Delta_i$ is a base of $R(2)$.
For $x=(x_i)\in\bigoplus_{i=1}^t\Z_{k_i}$, denote 
\begin{equation}
\lambda_x=(\lambda_{x_1}^1,\dots,\lambda_{x_t}^t)\in R_1^*\perp\dots\perp R_t^*=R^*,\label{Eq:lambdac}
\end{equation}
where $x_i$ is regarded as an element of $\{0,\dots,k_i-1\}$ and $\{\lambda_{j}^i\mid 1\le j\le k_i-1\}$ is the set of fundamental weights in $R_i^*$ with respect to $\Delta_i$.
The following lemma is immediate from the definition of $L_A(C)$.

\begin{lemma}\label{L:genA} For a generating set $\mathcal{C}$ of $C$, 
the set $\{\lambda_c\mid c\in\mathcal{C}\}$ and $R$ generate $L_A(C)$ as a lattice.
\end{lemma}

By \eqref{Eq:gammaj} and \eqref{Eq:innerij}, for $x=(x_i),y=(y_i)\in\bigoplus_{i=1}^t\Z_{k_i}$, we have
\begin{equation}
(\lambda_x| \lambda_x)=\sum_{i=1}^t\frac{x_i(k_i-x_i)}{k_i}\quad\text{and}\quad (\lambda_x| \lambda_y)\equiv-\sum_{i=1}^t\frac{x_iy_i}{k_i}\pmod\Z.\label{Eq:inner}
\end{equation}
Set 
\begin{equation}
\chi_\Delta=(\frac{\rho_{\Delta_1}}{k_1},\dots,\frac{\rho_{\Delta_t}}{k_t})\in \Q\otimes _\Z R\label{Eq:chi},
\end{equation}
where $\rho_{\Delta_i}$ is the Weyl vector of $R_i$ with respect to $\Delta_i$.
Note that $\chi_\Delta$ depends on $\Delta$.

Let $L_B(C)$ be the sublattice of $L_A(C)$ defined by 
\begin{equation}
L_B(C)=\{\alpha\in L_A(C)\mid (\alpha|\chi_\Delta)\in\Z\};\label{Eq:ConstB}
\end{equation}
we call $L_B(C)$ the lattice constructed by \emph{Construction B} from $C$.

\begin{remark} Let $\Delta'$ be a base of $R(2)$.
Then there exists an element $h$ in the Weyl group of $R$ such that $h(\Delta')=\Delta$.
Hence $h(\rho_{\Delta'})=\rho_\Delta$ and $h$ induces an isometry between lattices constructed by Construction B with respect to $\Delta'$ and $\Delta$.
Hence, up to isometry, the lattice $L_B(C)$ does not depend on a choice of a base $\Delta$.
\end{remark}

\begin{remark} When $k_1=\dots=k_t=2$, $L_A(C)$ and $L_B(C)$ are lattices constructed by Constructions A and B from a binary code $C$, respectively (cf. \cite[Chapter 7]{CS}).
\end{remark}

\begin{lemma}\label{Lem:enorm} Let $x=(x_i)\in \bigoplus_{i=1}^t\Z_{k_i}$.
Then $(\lambda_x|\lambda_x)\in2\Z$ if and only if $(\lambda_x|\chi_\Delta)\in\Z$.
\end{lemma}
\begin{proof} 
By \eqref{Eq:inner2} and \eqref{Eq:inner}, we have $$(\lambda_x|\lambda_x)=\sum_{i=1}^t\frac{x_i(k_i-x_i)}{k_i}={2}{}\sum_{i=1}^t\frac{(k_i-x_i)x_i}{2k_i}=2(\lambda_x|\chi_\Delta),$$
which proves this lemma. 
\end{proof}

Let $n$ be the least common multiple of $k_1,\dots,k_t$.
Note that for a simple root $\beta$ of $R_i(2)$, $\beta+L_B(C)$ has order $k_i$ in $L_A(C)/L_B(C)$.
Hence $L_A(C)/L_B(C)$ contains an  element of order $n$.
Since $\chi_\Delta\in (1/2n)R$ by \eqref{Eq:Weylvec}, we have $|L_A(C):L_B(C)|=n$ or $2n$.
By the definition of $L_B(C)$, we have the following:

\begin{lemma}\label{L:indexn}
$|L_A(C):L_B(C)|=n$ if and only if $\chi_\Delta\in (1/n)L_A(C)^*$.
\end{lemma}

Let $g_{\Delta_i}\in O(R_i)$ be defined as  in \eqref{Eq:g}.
Note that $g_{\Delta_i}$ belongs to the Weyl group of $R_i$.
For $e=(e_i)\in\bigoplus_{i=1}^t\Z_{k_i}$, set 
\begin{equation}
g_{\Delta,e}=((g_{\Delta_1})^{e_1},\dots,(g_{\Delta_t})^{e_t})\in O(L_A(C)).\label{Eq:gre}
\end{equation}

\begin{lemma}\label{L:gNc} $g_{\Delta,e}\in O(L_B(C))$ if and only if $\lambda_e\in L_A(C)^*$.
\end{lemma}
\begin{proof} 
It follows from \eqref{Eq:grho} that 
\begin{equation}
g_{\Delta,e}(\chi_\Delta)\in \chi_\Delta+(e_1\lambda_1^1,\dots,e_t\lambda_1^t)+R=\chi_\Delta+\lambda_e+R.\label{Eq:grhoc}
\end{equation}
By the definition of $L_B(C)$,  
$g_{\Delta,e}\in O(L_B(C))$ if and only if 
$\lambda_e\in L_A(C)^*$.
\end{proof}
By Lemma \ref{L:fpf}, we obtain the following:
\begin{lemma}\label{L:fpf2} The isometry $g_{\Delta,e}$ is fixed-point free and of order $n$ if and only if $\gcd(e_i,k_i)=1$ for all $1\le i\le t$
\end{lemma}

\subsection{Even lattices associated with codes over $\Z_p$}\label{Sec:Zp}

Let $p$ be an odd prime and set $\Z_p=\Z/p\Z$.
Let $\langle\cdot |\cdot \rangle:\Z_p^t\times\Z_p^t\to \Z_p$ be the canonical inner product of $\Z_p^t$; for $x=(x_i),y=(y_i)\in\Z_p^t$, $\langle x| y\rangle=\sum_{i=1}^t x_iy_i\pmod{p}$.
The \emph{Hamming weight} of an element $x=(x_i)\in\Z_p^t$ is $|\{i\mid x_i\neq0\}|$.
A subset $C$ of $\Z_p^t$ is called a \emph{code} of length $t$ over $\Z_p$ if $C$ is a subgroup of $\Z_p^t$.
A code $C$ over $\Z_p$ is said to be \emph{self-orthogonal} if $\langle c| c'\rangle=0$ for all $c,c'\in C$.
Note that a code $C$ is self-orthogonal if and only if $\langle c|c\rangle=0$ for all $c\in C$ since $p$ is an odd prime.

Let $R_i$ $(1\le i\le t)$ be a copy of the root lattice of type $A_{p-1}$.
Set $R=R_1\perp\dots\perp R_t.$
Then $\mathcal{D}(R)\cong\Z_p^t$.
For each $1\le i\le t$, fix a base $\Delta_i$ of the root system $R_i(2)$ of type $A_{p-1}$, and set $\Delta=\bigcup_{i=1}^t\Delta_i$.
Let $C$ be a code of length $t$ over $\Z_p$.
As in \eqref{Eq:ConstA}, we obtain the lattice $L_A(C)$ by Construction A from $C$.

By \eqref{Eq:inner}, for $x=(x_i),y=(y_i)\in\Z_p^t$, we have \begin{equation}
(\lambda_x|\lambda_x)=\frac{1}{p}\sum_{i=1}^tx_i(p-x_i)=\sum_{i=1}^t\left(x_i-\frac{x_i^2}{p}\right)\quad \text{and}\quad (\lambda_x|\lambda_y)\equiv -\frac{\langle x| y\rangle}{p}\pmod{\Z}.\label{Eq:innerp}
\end{equation}
Note that $\sum_{i=1}^tx_i(p-x_i)\in2\Z$.
Hence $(\lambda_x|\lambda_x)\in2\Z$ if and only if $\langle x| x\rangle=0.$ 
Thus we obtain the following proposition:

\begin{proposition}\label{P:so} 
The lattice $L_A(C)$ is even if and only if $C$ is self-orthogonal.
\end{proposition}

We now assume that $C$ is self-orthogonal.
It follows from $R\subset L_A(C)$ that $L_A(C)^*\subset R^*$.
Let $C^\perp$ denote the dual code of $C$, that is, $C^\perp=\{d\in \Z_p^t\mid \langle C| d\rangle=0\}$.
By \eqref{Eq:innerp} and Lemma \ref{L:genA}, we obtain the following:

\begin{lemma}\label{L:discA} $L_A(C)^*=L_A(C^\perp)$ and $|\mathcal{D}(L_A(C))|=|C^\perp/C|$.
\end{lemma}

Define $\chi_{\Delta}$ as in \eqref{Eq:chi}, i.e., 
\begin{equation}
\chi_\Delta=\frac{1}{p}(\rho_{\Delta_1},\dots,\rho_{\Delta_t})\in\frac{1}{p}R^*\label{Eq:chip}.
\end{equation}
It follows from $p\in2\Z+1$ that $\rho_{\Delta_i}\in R_i$, which shows  $\chi_{\Delta}\in(1/p)R$.

Let $L_B(C)$ be the sublattice of $L_A(C)$ defined as in \eqref{Eq:ConstB}.
Since $L_A(C)$ is even, so is $L_B(C)$.
Fix a simple root $\alpha\in\Delta$.
Then $\Delta\subset \alpha+L_B(C)$ since $\alpha-\beta\in L_B(C)$ for any $\beta\in \Delta$.
By Lemma \ref{L:indexn} and $\chi_\Delta\in(1/p)R\subset (1/p)L_A(C)^*$, we have $|L_A(C):L_B(C)|=p$.
It follows from $(\chi_\Delta| \alpha)=1/p$ that $L_A(C)/L_B(C)=\langle \alpha+L_B(C)\rangle$.

\begin{lemma}\label{L:discB} $|\mathcal{D}(L_B(C))|=p^2|C^\perp/C|$ and $L_B(C)^*=L_A(C^\perp)+\Z\chi_\Delta$.
\end{lemma}
\begin{proof}
Since $|L_B(C)^*{:}L_A(C)^*|=|L_A(C){:}L_B(C)|=p$, we obtain $|\mathcal{D}(L_B(C))|=p^2|C^\perp/C|$ by lemma \ref{L:discA}.
By the definition of $L_B(C)$, we have $\chi_\Delta\in L_B(C)^*$.
Since $\chi_\Delta\notin L_A(C)^*$, we have the result.
\end{proof}

Let $c\in C$.
Since $L_A(C)$ is even, we have $(\lambda_c|\lambda_c)\in2\Z$.
By Lemma \ref{Lem:enorm}, $(\lambda_c|\chi_\Delta)\in\Z$, and $\lambda_c\in L_B(C)$.
In addition, for $c'\in C$, $\lambda_c+\lambda_{c'}\in\lambda_{c+c'}+L_B(C)$.
It follows from $L_A(\{\allzero\})=R\subset L_A(C)$ that $L_B(\{\allzero\})=\{\alpha\in R\mid (\alpha|\chi_\Delta)\in\Z\}\subset L_B(C).$
Hence we obtain the following lemma.

\begin{lemma}\label{L:genB}
For a generating set $\mathcal{C}$ of $C$, the set $\{\lambda_c\mid c\in \mathcal{C}\}$ and $L_B(\{\allzero\})$ generate $L_B(C)$ as a lattice.
\end{lemma}

\begin{remark}\label{R:CC'} Let $C$ and $C'$ be codes over $\Z_p$.
Assume that $C$ and $C'$ are equivalent, that is, there exists $f_0\in \langle \varepsilon_i\mid 1\le i\le t\rangle{:}S_n$ such that $f_0(C)=C'$, where $\varepsilon_i$ acts on $\Z_p^t$ as $-1$ on the $i$-th component and as $1$ on the other components and the symmetric group $S_n$ acts on the coordinates as permutations.
Note that the diagram automorphism of $\Delta_i$ of order $2$ induces $\varepsilon_i$ on $R_i^*/R_i\cong \Z_p$.
Then $f_0$ induces an isometry $f:L_A(C)\to L_A(C')$.
Since $f$ fixes $\chi_\Delta$, we have $L_B(C)\cong L_B(C')$.
\end{remark}

\subsection{Characterizations of even lattices associated with codes over $\Z_p$}\label{Sec:ChaAB}
In this subsection, we give characterizations of the even lattices $L_A(C)$ and $L_B(C)$.

Let $p$ be an odd prime and let $C$ be a self-orthogonal code of length $t$ over $\Z_p$.
By the construction, $L_A(C)$ contains $R=L_A(\{\allzero\})$, which is isometric to the root lattice of type $A_{p-1
}^t$, as a full sublattice.

Conversely, if an even lattice $L$ contains a full sublattice $S$ isometric to the root lattice of type $A_{p-1}^t$, then there exists a self-orthogonal code $C$ of length $t$ over $\Z_p$ such that $L\cong L_A(C)$.
Indeed, $C=L/S\subset S^*/S\cong\Z_p^t$.
Hence we obtain the following proposition.

\begin{proposition}\label{P:ChaA} Let $L$ be an even lattice.
Then $L\cong L_A(C)$ for some self-orthogonal code $C$ of length $t$ over $\Z_p$ if and only if $L$ contains the root lattice of type $A_{p-1}^t$ as a full sublattice.
\end{proposition}

We will also discuss a characterization of the lattice $L_B(C)$.

\begin{proposition}\label{P:ChaB} Let $C$ be a self-orthogonal code over $\Z_p$.
Assume that $(L_A(C))(2)=R(2)$.
Let $L$ be a full sublattice of $L_A(C)$ with index $p$.
Then $L\cong L_B(C)$ if and only if $L$ is rootless.
\end{proposition}
\begin{proof}
By the definition of $L_B(C)$ and Lemma \ref{L:modp}, the assumption $(L_A(C))(2)=R(2)$ implies that $L_B(C)$ is rootless.

Assume that $L$ is rootless.
Let $\gamma\in R(2)$ be a root.
Then $\gamma\in L_A(C)\setminus L$.
Since $p$ is a prime, we have
\begin{equation}
L_A(C)/L=\langle \gamma+L\rangle.\label{Eq:al}
\end{equation}
Hence $L_A(C)=\bigcup_{i=0}^{p-1}(i\gamma+L)$, and by $L(2)=\emptyset$, 
\begin{equation}
\bigcup_{i=1}^{p-1}(i\gamma+L)(2)=R(2).\label{Eq:iv}
\end{equation}

Let $1\le i\le p-1$ and let $\alpha,\beta\in (i\gamma+L)(2)$ with $\alpha\neq \beta$.
If $(\alpha|\beta)=1$, then $\alpha-\beta\in L(2)$, which contradicts that $L(2)=\emptyset$.
If $(\alpha|\beta)=-2$, then $\alpha=-\beta$ and $p=2$, which contradicts that $p$ is an odd prime.
Hence $(\alpha|\beta)\in \{0,-1\}$.
By Lemma \ref{Lem:Ap1}, $|(i\gamma+L)(2)|\le pt$.
It follows from $|(L_A(C))(2)|=|R(2)|=p(p-1)t$ and \eqref{Eq:iv} that $|(i\gamma+L)(2)|= pt$ for all $i$.
By Lemma \ref{Lem:Ap1}, $(\gamma+L)(2)$ consists of  a base of the root system $R(2)$ of type $A_{p-1}^t$ and the negated highest roots; we fix a base $\Delta\subset(\lambda+L)(2)$ such that $\gamma\in \Delta$.
Then $\sum_{\alpha\in\Delta} c_\alpha\alpha\in R$ belongs to $L$ if and only if $\sum_{\alpha\in\Delta}c_\alpha\equiv0\pmod{p}$, equivalently, the inner product $(\chi_\Delta|\sum_{\alpha\in\Delta} c_\alpha\alpha)\in\Z$, where $\chi_\Delta$ is the vector defined in \eqref{Eq:chip} with respect to $\Delta$.
Hence $L_B(\{\allzero\})=\{\beta\in R\mid (\beta| \chi_\Delta)\in\Z\}\subset L$.

For $c\in C$, let $\lambda_c\in L_A(C)$ defined as in \eqref{Eq:lambdac}.
By Lemma \ref{Lem:enorm}, $\lambda_c\in L_B(C)$.
Set $$C_0=\{c\in C\mid ({\lambda}_c+L_B(\{\allzero\}))\cap L\neq\emptyset\}.$$
If $C=C_0$, then $L=L_B(C)$ by Lemma \ref{L:genB}; we now assume that $C\neq C_0$.
It follows from $L_B(\{\allzero\})\subset L$ that $L_B(C_0)\subset L$.

First, we show that $|C:C_0|=p$.
Let $d,d'\in C\setminus C_0$.
By \eqref{Eq:al}, there exist $1\le i,i'\le p-1$ such that $\lambda_d+i\gamma,\lambda_{d'}+i'\gamma\in L$, where $\gamma\in\Delta$ with $(\gamma|\chi_\Delta)=1/p$.
Let $1\le j\le p-1$ such that $i+ji'\equiv 0\pmod p$.
Then $$\lambda_d+i\gamma+j(\lambda_{d'}+i'\gamma)\in (\lambda_{d+jd'}+L_B(\{\allzero\}))\cap L.$$
Hence $d+jd'\in C_0$, and $|C:C_0|=p$.

Let $e=(e_i)\in C_0^\perp\setminus C^\perp$ and let $g_{\Delta,e}
\in O(L_A(C))$ as in \eqref{Eq:gre}.
Let us show that some power of $g_{\Delta,e}$ sends $L$ to $L_B(C)$.

Let $c\in C_0$.
Note that for $1\le i\le t$ and simple roots $\alpha_j^i$ of $R_i$, we have $(\sum_{j=1}^q\alpha_j^i| \rho_{\Delta_i})= (q\alpha_1^i|\rho_{\Delta_i})=q$.
Hence, by \eqref{Eq:grho}, we have 
\begin{equation}
g_{\Delta,e}(\lambda_c)\in \lambda_c-(e_1c_1\alpha_1^1,\dots,e_tc_t\alpha_1^t)+L_B(\{\allzero\}).\label{Eq:grep}
\end{equation}
It follows from $\langle c| e\rangle=0$ that $\sum_{i=1}^t c_ie_i\equiv0\pmod{p}$, which shows $(e_1c_1\alpha_1^1,\dots,e_tc_t\alpha_1^t)\in L_B(\{\allzero\})$.
Thus $g_{\Delta,e}(\lambda_c+L_B(\{\allzero\}))=\lambda_c+L_B(\{\allzero\})$.

Recall that $L/L_B(C_0)=\langle {\lambda}_d+j\gamma+L_B(C_0)\rangle$ for some $d\in C\setminus C_0$ and $1\le j\le p-1$.
Take $1\le s\le p-1$ with $-s\langle c| d\rangle+j\equiv 0\pmod{p}$.
By \eqref{Eq:grho} and \eqref{Eq:gre}, we have $$(g_{\Delta,e})^s(\lambda_d+j\gamma)\in \lambda_d-s(e_1d_1 \alpha_1^1,\dots,e_td_t\alpha_1^t)+(g_{\Delta,e})^s(j\gamma)+L_B(\{\allzero\}).$$
Since $$((g_{\Delta,e})^s(\lambda_d+j\gamma)| \chi_\Delta)\equiv \frac{-s\langle c| d\rangle+j}{p}\equiv0\pmod\Z,$$
we have 
$(g_{\Delta,e})^s(\lambda_d+j\gamma+L_B(\{\allzero\}))=\lambda_d+L_B(\{\allzero\})$.
Therefore, by Lemma \ref{L:genB}, the isometry $(g_{\Delta,e})^s\in O(L_A(C))$ sends $L$ to $L_B(C)$, which proves this proposition.
\end{proof}

Now, let us give another characterization of the lattices $L_A(C)$ and $L_B(C)$.

\begin{theorem}\label{T:ChaB} Let $L$ be a rootless even lattice of rank $m$ and let $p$ be an odd prime.
Assume that there exist $\lambda+L\in \mathcal{D}(L)$ with $p\lambda\in L$ and a fixed-point free isometry $g\in O(L)$ satisfying the following:
\begin{enumerate}[{\rm (a)}]
\item $|N(2)|=pm$, where $N={\rm Span}_\Z\{\lambda,L\}$;
\item $g(\lambda+L)=\lambda+L$.
\end{enumerate}
Then there exists a self-orthogonal code $C$ of length $m/(p-1)$ over $\Z_p$ such that $N\cong L_A(C)$ and $L\cong L_B(C)$.
Here $g$ corresponds to $g_{\Delta,e}$ defined in \eqref{Eq:gre} with respect to some $e\in C^\perp$ of Hamming weight $m/(p-1)$ and some base $\Delta$ of $N(2)$.
\end{theorem}
\begin{proof} 

First, we will show that $(\lambda+L)(2)$ contains a base of the root system of type $A_{p-1}$.
Let $\gamma \in (\lambda+L)(2)$.
By (b), $g$ acts on $(\lambda+L)(2)$.
Since this action is fixed-point free, we have $g^i(\gamma) \in (\lambda+L)(2)$ and  $\gamma \neq g^i(\gamma) $ for any $1\le i\le p-1$.
Since both $\gamma$ and $g^i(\gamma)$ are roots and $p\neq2$, we have $(\gamma| g^i(\gamma))\in\{0,\pm1\}$ for $1\le i\le p-1$.
If $(\gamma| g^i(\gamma))= 1$, then $(1-g^i)(\gamma)\in L(2)$, which contradicts that $L(2)=\emptyset$. 
Hence $(\gamma|g^i(\gamma))\in\{-1,0\}$.
It follows from $\sum_{i=0}^{p-1}g^i(\gamma)=0$ that
$$\sum_{i=1}^{p-1} (\gamma| g^i(\gamma)) =-(\gamma|\gamma) =-2.$$ 
Then, there exists $1 \leq j \leq p-1$ such that 
\[
(\gamma| g^j(\gamma))= (\gamma| g^{p-j}(\gamma))=-1 \text{ and } (\gamma| g^q(\gamma)) =0 \text { for } q\neq j, p-j.
\] 
Hence $\{g^i(v)\mid 0\le i\le p-1\}$ is the union of a base and the negated highest root of type $A_{p-1}$; $X_1={\rm Span}_\Z\{ g^i(v)\mid 0\le i\le p-1\}$ is isometric to the root lattice of type $A_{p-1}$. 
For $1\le s\le p-1$, $g(s\lambda+L)=s\lambda+L$, and  by the same argument, $(s\lambda+L)\cap X_1(2)$ is also the union of a base and the negated highest root of the root system $X_1(2)$ of type $A_{p-1}$.

Assume that $N(2)\neq X_1(2)$; let $\beta \in N(2)\setminus X_1$. 
Let us show that $\beta \perp X_1$.
Note that $\beta\in (s\lambda+L)(2)$ for some $1\le s\le p-1$.
Let $\alpha\in (s\lambda+L)\cap X_1(2)$.
By the same argument as above, we have $(\alpha|\beta)\in\{0,-1\}$.
If $(\alpha|\beta)=-1$, then by $\sum_{i=0}^{p-1}g^i(\alpha)=0$ and $(g^i(\alpha)|\beta)\in\{0,\pm1\}$, there exists $1\le j\le p-1$ such that $(g^j(\alpha)| \beta) =1$, and hence $g^j(\alpha) -\beta\in L(2)$, which is a contradiction.
Thus, $\beta \perp ((s\lambda+L)\cap X_1)(2)$.
Since $X_1$ is spanned by $((s\lambda+L)\cap X_1)(2)$, we have $\beta\perp X_1$.

Set $X_2={\rm Span}_\Z\{g^i(\beta)\mid 0\le i\le p-1\}$.
Then by the same argument as in the case of $X_1$, the even lattice $X_2$ is isometric to the root lattice of type $A_{p-1}$.
By recursion and the finiteness of $N(2)$, the set $N(2)$ is the orthogonal sum of $t$ copies of the root system of type $A_{p-1}$ for some $t$.
Hence $|N(2)|=tp(p-1)$.
On the other hand, by (a), we have $|N(2)|=pm$.
Hence $t=m/(p-1)$.
In particular, $N$ contains $A_{p-1}^t$ as a full sublattice.
Since $|N:L|=p$ and $N\setminus L$ has a root, $N$ is even.
By Proposition \ref{P:ChaA}, $N\cong L_A(C)$ for some self-orthogonal code $C$ of length $t$ over $\Z_p$.
Since $L$ is a rootless index $p$ sublattice of $N$, we have $L\cong L_B(C)$ by Proposition \ref{P:ChaB}. 
Let $\Delta$ be a base of $N(2)$ corresponding to the definition of $L_B(C)$ and let $\tilde{\Delta}$ be the union of $\Delta$ and the negated highest root.
Then $\tilde\Delta\subset s\lambda+L\in N/L$ for some $1\le s\le p-1$.
By (b), $g$ preserves every $\tilde\Delta_i\subset\tilde\Delta$.
By Lemmas \ref{L:dih} and \ref{L:fpf2}, $g=g_{\Delta,e}$ for some $e\in(\Z_p^\times)^t$, and by Lemmas \ref{L:gNc} and \ref{L:discA}, $e\in C^\perp$.
Note that the Hamming weight of $e$ is $t=m/(p-1)$. 
\end{proof}

\begin{remark} It was proved in \cite[Proposition 1.8]{Sh04} that a rootless even lattice $L$ of rank $m$ can be constructed by Construction B from a doubly even binary code if and only if there exists $\lambda+L\in \mathcal{D}(L)$ such that $2\lambda\in L$ and $|(\lambda+L)(2)|=2m$.
Note that any lattice contains the $-1$-isometry, the fixed-point free isometry of order $2$.
\end{remark}

\subsection{Coinvariant lattices of the Leech lattice and Construction B}\label{S:example}
In this subsection, we describe some coinvariant lattices of the Leech lattice $\Lambda$ as even lattices constructed by Construction B.
We adopt the notations of \cite{ATLAS} (cf.\ \cite{HL90}) for conjugacy classes of $O(\Lambda)$.

Let $\Lambda$ be the Leech lattice.
For an element $h$ in the conjugacy class $pX$ of $O(\Lambda)$, we denote by $\Lambda_{pX}$ the associated  coinvariant lattice $\Lambda_h$ (see \eqref{Eq:L_g} for the definition).
Note that the action of $h$ on $\Lambda_{pX}$ is fixed-point free.
By using Theorem \ref{T:ChaB} and MAGMA \cite{MAGMA}, one can verify that there exists a self-orthogonal code $C$ over $\Z_p$ such that $\Lambda_{nX}\cong L_B(C)$ if $nX\in\{3B,3C,5B,5C,7B\}$.
We summarize in Table \ref{T:example} the rank of $\Lambda_{pX}$, the root system of $L_A(\{\allzero\})(2)$, the discriminant group $\mathcal{D}(\Lambda_{pX})$, the length $t$ of $C$ and $\dim C$.
Note that $\dim C$ is determined by $\mathcal{D}(\Lambda_{pX})$ (see Lemma \ref{L:discB}).

\begin{longtable}[c]{|c|c|c|c|c|} 
\caption{Some coinvariant lattices of the Leech lattices} \label{T:example}\\
\hline 
Conjugacy class  & rank $\Lambda_{pX}$& root system & $\mathcal{D}(\Lambda_{pX})$ &$[p,t,\dim C]$ \\ \hline \hline 
$3B$&$12$&$A_2^6$&$\Z_3^6$&$[3,6,1]$ \\ \hline
$3C$&$18$&$A_2^9$&$\Z_3^5$&$[3,9,3]$ \\ \hline
$5B$&$16$&$A_4^4$&$\Z_5^4$&$[5,4,1]$ \\ \hline 
$5C$&$20$&$A_4^5$&$\Z_5^3$&$[5,5,2]$ \\ \hline 
$7B$&$18$&$A_6^3$&$\Z_7^3$&$[7,3,1]$ \\ \hline 
\end{longtable}

A code $C$ over $\Z_p$ with the parameter in Table \ref{T:example} is unique up to equivalence if $L_B(C)$ is rootless as follows; see Remark \ref{R:CC'} for the equivalence of codes over $\Z_p$.

\begin{lemma}\label{L:uniqueC} Let $C$ be a self-orthogonal code of length $t$ over $\Z_p$.
Assume that $L_B(C)$ is rootless.
\begin{enumerate}[{\rm (1)}]
\item If $[p,t,\dim C]=[3,6,1]$, then $C$ is equivalent to $\langle (1,1,1,1,1,1)\rangle_{\Z_3}$.
\item If $[p,t,\dim C]=[3,9,3]$, then $C$ is equivalent to the code over $\Z_3$ with generator matrix
$$\begin{pmatrix}1&1&1&1&1&1&1&1&1\\ 1&1&1&2&2&2&0&0&0\\
1&2&0&1&2&0&1&2&0\end{pmatrix}.$$
\item If $[p,t,\dim C]=[5,4,1]$, then $C$ is equivalent to $\langle (1,1,2,2)\rangle_{\Z_5}$.
\item If $[p,t,\dim C]=[5,5,2]$, then $C$ is equivalent to the code over $\Z_5$ with generator matrix
$$\begin{pmatrix}1&1&1&1&1\\ 1&2&4&3&0\end{pmatrix}.$$
\item If $[p,t,\dim C]=[7,3,1]$, then $C$ is equivalent to $\langle (1,2,3)\rangle_{\Z_7}$.
\end{enumerate}
\end{lemma}
\begin{proof} Since $L_B(C)$ is rootless, $C$ has no  codewords of Hamming weight $3$ and $2$ if $p=3$ and $p=5$, respectively.
One can directly check that there is only one possible code with given dimension and length, up to equivalence.
\end{proof}

\begin{remark} The dual codes of the codes in Lemma \ref{L:uniqueC} contains a codeword of Hamming weight $t$.
\end{remark}

\begin{remark}\label{R:2A} For the conjugacy classes $2A$ and $-2A$ of $O(\Lambda)$, the coinvariant lattices are $\sqrt2E_8$ and the Barnes-Wall lattice of rank $16$, respectively.
These lattices are constructed by Construction B from the binary codes $\langle (1^8)\rangle_{\Z_2}$ and the Reed-Muller code $RM(1,4)$, respectively.
\end{remark}

\section{Cyclic orbifolds of lattice VOAs having extra automorphisms } \label{extraauto}
Let $L$ be a positive-definite rootless even lattice of rank $m$. 
Let $g\in O(L)$ be a fixed-point free isometry of odd prime order $p$.
We will classify $L$ such that $V_L^{\hat{g}}$ has extra automorphisms.

\begin{remark}\label{R:Sh04a} If a fixed-point free isometry of $L$ has order $2$, then it must be the $-1$-isometry.
Let $\theta\in O(\hat{L})$ be a lift of the $-1$-isometry of $L$.
For a rootless even lattice $L$, $V_L^{\theta}$ has an extra automorphism if and only if $L\cong L_B(C)$ for some doubly even binary code $C$ (\cite[Theorem 3.15]{Sh04}).
\end{remark}

Assume that $V_L^{\hat{g}}$ has an extra automorphism $\tau$.
Then one of the following holds (cf. Definition \ref{D:untw}):
\begin{enumerate}[(I)]
\item {\bf $V_L(1)\circ\tau$ is of untwisted type}; notice that 
\begin{equation*}
V_L(1)\circ\tau\not\cong V_L(r)\quad {\rm for\ all}\ 1\le r\le p-1.\label{Eq:ext}
\end{equation*}
Otherwise, by Lemma \ref{Lem:conj} (2) and \eqref{Eq:fusionun}, we have 
$$\{V_L(i)\circ\tau\mid 0\le i\le p-1\}=\{(V_L(1)\circ\tau)^{\boxtimes i}\mid 0\le i\le p-1\}=\{V_L(i)\mid 0\le i\le p-1\},$$ and $\tau$ is not extra by Proposition \ref{P:stabVL1}.
Since $V_L(1)$ is a simple current module, so is $V_L(1)\circ\tau$.
Hence, by the discussions in Section \ref{S:un}, 
$$V_L(1)\circ\tau\cong V_{\lambda+L}(r)$$
for some $\lambda+L\in \mathcal{D}(L)\setminus\{L\}$ with $(1-g)\lambda\in L,\ 0\le r\le p-1$. 
\medskip

\item {\bf $V_L(1)\circ\tau$ is of twisted type};   
in this case, $$V_L(1)\circ \tau\cong V_{L}^T[\hat{g}^s](j)$$ for some irreducible $\hat{g}^s$-twisted module $V_{L}^T[\hat{g}^s]$ with $1\le s\le p-1$ and $0\le j\le p-1$. 
 
\end{enumerate}
We will deal with Cases (I) and (II) in Sections \ref{S:1} and \ref{S:7}, respectively.  
\begin{remark}
We note that Cases (I) and (II) are not exclusive. For some special lattices, both (I) and (II) occur, for example, $\sqrt{2}E_8$, the Barnes-Wall lattice of rank $16$ and Coxeter-Todd lattice of rank $12$ (\cite{Sh04,CLS}).   
\end{remark}

\section{Case (I):   $V_L(1)\circ\tau \cong V_{\lambda+L}(r)$ for some $\lambda+L\in \mathcal{D}(L)\setminus\{L\}$ and $\tau \in \Aut(V_L^{\hat{g}})$}\label{S:1}

In this section, we classify even lattices $L$ and isometries $g$ of odd prime order satisfying Case I in Section \ref{extraauto}.

\subsection{Necessary conditions on even lattices for Case (I)}

For latter use, we consider a slightly general setting; there exist $\tau\in \Aut(V_L^{\hat{g}})$, $0\le r\le p-1$, $1\le j\le p-1$ and $\lambda+L\in \mathcal{D}(L)\setminus \{L\}$ with $(1-g)\lambda\in L$ such that $$V_L(j)\circ\tau\cong V_{\lambda+L}(r).$$
Then $\dim V_{L}(j)_1=m/{(p-1)}$ and $\dim V_{\lambda+L}(r)_1=|(\lambda+L)(2)|/p$ are equal, namely, 
\begin{equation*}
|(\lambda+L)(2)|=\frac{pm}{p-1}.\label{Eq:lambda2}
\end{equation*}
It follows from \eqref{Eq:fusionun} that $V_{\lambda+L}(r)^{\boxtimes p}=V_{p\lambda+L}(r')$ for some $0\le r'\le p-1$.
Since $V_L(j)^{\boxtimes p}=V_L(0)$, we have $V_{\lambda+L}(r)^{\boxtimes p}=V_L(0)$.
Hence $p\lambda\in L$.
In addition, for $1\le q\le p-1$, there exist $1\le j_q\le p-1$ and $0\le r_q\le p-1$ such that 
\begin{equation}
V_L(j_q)\circ \tau\cong V_{q\lambda+L}(r_q).\label{Eq:conj}
\end{equation}
By the same argument as above, we have$$|(q\lambda+L)(2)|=\frac{pm}{p-1}.$$
Set $N={\rm Span}_\Z\{\lambda,L\}$.
Then by $L(2)=\emptyset$, we have $$|N(2)|=\sum_{i=1}^{p-1}|(i\lambda+L)(2)|=pm.$$
By our assumption that $V_L(j)\circ\tau\cong V_{\lambda+L}(r)$, $V_{\lambda+L}$ is $\hat{g}$-stable and we have $g(\lambda+L)=\lambda+L$.
Thus, by Theorem \ref{T:ChaB}, we obtain the following:

\begin{proposition}\label{P:(I)} Let $L$ be a rootless even lattice of rank $m$ and let $g\in O(L)$ be a fixed-point free isometry of odd prime order $p$.
Let $\hat{g}\in O(\hat{L})$ be a lift of $g$. 
Suppose there exist $\tau\in\Aut(V_L^{\hat{g}})$, $1\le j\le p-1$, $0\le r\le p-1$ and $\lambda\in \mathcal{D}(L)\setminus\{L\}$ with $(1-g)\lambda\in L$ such that $V_L(j)\circ\tau\cong V_{\lambda+L}(r)$. 
Then $L\cong L_B(C)$ for some self-orthogonal code $C$ of length $m/(p-1)$ over $\Z_p$.   Moreover, $g$ corresponds to $g_{\Delta,e}$ defined in \eqref{Eq:gre} with respect to some $e\in C^\perp$ of Hamming weight $m/(p-1)$ and some base $\Delta$ of $N(2)$. In particular, 
$C^\perp$ contains a codeword of Hamming weight $m/(p-1)$.
\end{proposition}

\subsection{Extra automorphisms of cyclic orbifolds of lattice VOAs}\label{S:extra}
In this subsection, we prove that the cyclic orbifold $V_L^{\hat{g}}$ has extra automorphisms if the even lattice $L$ can be  constructed by Construction B from a subgroup of $\bigoplus_{i=1}^t\Z_{k_i}$ as in \eqref{Eq:ConstB} and  $g=g_{\Delta,e}$ is a fixed-point free isometry as in \eqref{Eq:gre}.
Notice that $L$ is not necessarily rootless and that $|g|$ is not necessarily a prime throughout this subsection.

\subsubsection{Inner automorphisms of the simple Lie algebra of type $A_{k-1}$}\label{S:autoLie}
Let $k\in\Z_{\ge2}$ and let $\g$ be a simple Lie algebra of type $A_{k-1}$ over $\C$.
In this subsection, we discuss three inner automorphisms of $\g$ of finite order. 
The following is a standard model for $\g$:
$$\g=\{X\in M_{k}(\C)\mid \tr(X)=0\},$$
where the Lie bracket $[\ ,\ ]$ on $\g$ is given by $[X,Y]=XY-YX$.
Set$$F={\rm diag}(\omega,\omega^{2},\dots,\omega^{(k-1)},1)\in  M_{k}(\C),$$
where $\omega=\exp({2\pi\sqrt{-1}/k})$.
We also set  
\[ 
G=\begin{pmatrix}
0&1  &\cdots &0\\
\vdots&\ddots&\ddots & \vdots\\
0&  0&\ddots& 1\\
1 &0  & \cdots & 0
\end{pmatrix}, \quad 
Z=\frac{1}{\sqrt{k}}%%
\begin{pmatrix}
  \omega&\omega^{2}&\cdots & \omega^{(k-1)} &1\\
  \omega^{2}&\omega^{4}&\cdots &\omega^{2(k-1)} &1 \\
   \vdots&\ddots&\ddots&\vdots&\vdots\\
   \omega^{(k-1)}&\omega^{2(k-1)}&\ddots &\omega^{(k-1)^2} &1\\
   1&1&\cdots& 1&1
\end{pmatrix}\in  M_{k}(\C).
\]
Then $F^{k}=G^{k}=I$. 
Moreover $$Z^2=\begin{pmatrix}
0&0  &\cdots&0 &1&0\\
0&0&\cdots&1&0&0\\
\vdots&\vdots&\ddots &\vdots&\vdots& \vdots\\
0&  1&\cdots&0&0& 0\\
1&  0&\cdots&0&0& 0\\
0 &0  & \cdots&0&0 & 1
\end{pmatrix}$$
and $Z^2=I$ if $k=2$, and $Z^4=I$ if $k>2$.
Furthermore, we obtain 
\begin{equation}
Z^{-1}GZ=F,\quad G^{-1}FG=\omega^{-1} F,\quad Z^{-2}GZ^2=G^{-1}.
\label{Eq:SFG}
\end{equation}

Let us consider the following three inner automorphisms of $\g$: for $X\in\g$,
\begin{align*}
\varphi_{\g}:X\mapsto F^{-1}XF,\qquad \hat{g}_{\g}:X\mapsto G^{-1}XG,\qquad \zeta_{\g}:X\mapsto Z^{-1}XZ.
\end{align*}
By \eqref{Eq:SFG}, we have $\zeta_{\g}(G)=F$, $\zeta_{\g}(F)=G^{-1}$ and $\hat{g}_\g^{-1}\varphi_\g\hat{g}_\g=\varphi_\g$.
The following is also obtained by \eqref{Eq:SFG}:
\begin{lemma}\label{conjhes}
\begin{enumerate}[{\rm (1)}]
\item The orders of both $\varphi_{\g}$ and $\hat{g}_{\g}$ are $k$.
\item $\zeta_{\g}^{-1} \hat{g}_{\g}\zeta_{\g}=\varphi_{\g}$ and $\zeta^{-1}_{\g}\varphi_{\g}\zeta_{\g}=\hat{g}_{\g}^{-1}$ on $\g$. 
In particular, $\zeta_{\g}$ normalizes the subgroup generated by $\varphi_{\g}$ and $\hat{g}_{\g}$.
\item $\varphi_{\g}$ and $\hat{g}_{\g}$ are commutative on $\g$.
\end{enumerate}
\end{lemma}

We now describe $\varphi_{\g}$ and $\hat{g}_{\g}$ on $\g$ explicitly.
Let $E_{ij}\in M_{k}(\C)$ denote the matrix with $1$ at the $(i,j)$-entry and $0$ otherwise.
Let $\h$ be the Cartan subalgebra of $\g$ consisting of all diagonal matrices with trace $0$, that is, $\h={\rm Span}_\C\{E_{ii}-E_{i+1,i+1}\mid 1\le i\le k-1\}$.
We identify $\mathfrak{h}$ with its dual $\mathfrak{h}^*$ via the normalized Killing form $(\cdot | \cdot )$ so that $(\alpha|\alpha)=2$ for any root $\alpha\in \mathfrak{h}^*$.
Take simple roots $\alpha_i$ $(1\le i\le k-1)$ so that $\alpha_i\in \Q_{>0}(E_{ii}-E_{i+1,i+1})$.
For $1\le i<j\le k$, the root space with respect to the root $\sum_{s=i}^{j-1}\alpha_s$ is $\C E_{ij}$.

Let $A$ be the root lattice of type $A_{k-1}$ spanned by the base $\Delta=\{\alpha_i\mid1\le i\le k-1\}$.
By the definition of the matrix $G$, we have $\hat{g}_{\g}(\alpha_{i})=\alpha_{i+1}$ if $1\le i\le k-2$ and $\hat{g}_{\g}(\alpha_{k-1})=-\sum_{i=1}^{k-1}\alpha_i$, the negated highest root.
Thus the restriction of $\hat{g}_\g$ on $\h$ is the isometry $g_{\Delta}$ as in \eqref{Eq:g}, that is, $\hat{g}_{\g}$ is a lift of $g_{\Delta}$ to $\g$.

Recall that $\varphi_{\g}=id$ on $\h$ and $\varphi_{\g}$ acts as $\omega^{j-i}$ on $\C E_{i,j}$.
Let $\rho_{\Delta}$ be the Weyl vector as in \eqref{Eq:Weylvec}.
We view $\rho_{\Delta}$ as an element of $\h$.
Then $\ad\; (\rho_{\Delta})$ acts as $0$ on $\h$ and as $(j-i)$ on $\C E_{i,j}$.
Hence we have 
\begin{equation}
\varphi_{\g}= 
\exp\left(\frac{1}{k}\left(2\pi \sqrt{-1} \ad~  (\rho_{\Delta})\right)\right).\label{Eq:varphi}
\end{equation}

\subsubsection{Extra automorphisms of cyclic orbifolds of lattice VOAs of type $A_{k-1}$}\label{S:ext}
Let $A$ be the root lattice of type $A_{k-1}$ and let $V_{A}$ be the associated lattice VOA.
Then $\g=(V_{A})_1$ is a simple Lie algebra of type $A_{k-1}$.

Recall that $V_{A}$ is generated by $\g$ and that $\Aut(V_A)=\Aut(\g)$.
We extend the inner automorphisms $\varphi_{\g}$, $\hat{g}_\g$ and $\zeta_\g$ of $\g$ in Section \ref{S:autoLie} to those of $V_{A}$, for which we use the same symbols.
By Lemma \ref{conjhes}, $\varphi_{\g}$ and $\hat{g}_{\g}$ are commuting automorphisms of order $k$ on $V_{A}$ and $\zeta_{\g}$ normalizes the subgroup $\langle \varphi_{\g},\hat{g}_{\g}\rangle$ of $\Aut(V_{A})$.

Let $\Delta$ be the base of $A$ as in Section \ref{S:autoLie} and let $A_0=\{\beta\in A\mid (\beta|\dfrac{{\rho_{\Delta}}}{k})\in \Z\}.$
It follows from $\rho_{\Delta}\in A^*$ and $(\rho_{\Delta}|\alpha_i)=1$ for $1\le i\le k-1$ that $A_0$ is an index $k$ sublattice of $A$.
By \eqref{Eq:varphi}, the fixed-point subVOA $V_{A}^{\varphi_{\g}}$ of $\varphi_\g$ is $V_{A_0}$.
By Lemma \ref{conjhes} (2), $\zeta_{\g}(V_{A}^{\hat{g}_{\g}})=V_{A}^{\varphi_{\g}}$ and $\zeta_{\g}(V_{A}^{\varphi_{\g}})=V_{A}^{\hat{g}_{\g}}$.
Hence $\zeta_{\g}$ is an automorphism of $V_{A}^{\langle \varphi_{\g},\hat{g}_{\g}\rangle}=V_{{A_0}}^{\hat{g}_{\g}}$ such that $V_{A_0}\circ\zeta_{\g}\cong V_{A}^{\hat{g}_{\g}}$ as $V_{A_0}^{\hat{g}_{\g}}$-modules.
By Proposition \ref{P:stabVL1}, $\zeta_\g$ is an extra automorphism of $V_{A_0}^{\hat{g}_\g}$.

\subsubsection{Extra automorphisms of cyclic orbifolds of lattice VOAs: general case}

Let $t\in\Z_{>0}$.
For $1\le i\le t$, let $k_i\in\Z_{\ge2}$ and let $R_i$ be the root lattice of type $A_{k_i-1}$.
Let $\Delta_i$ be a base of $R_i(2)$.
Set $R=\bigoplus_{i=1}^tR_{i}$.
Then $\Delta=\bigcup_{i=1}^t\Delta_i$ is a base of $R(2)$.
The vector $\chi_{\Delta}\in \Q\otimes_\Z R$ is defined as in \eqref{Eq:chi}. 

Let $C$ be a subgroup of $\mathcal{D}(R)\cong\bigoplus_{i=1}^t\Z_{k_i}$ such that $N=L_A(C)$ is even.
Set $L=L_B(C)$.
Let $e\in \bigoplus_{i=1}^t\Z_{k_i}^\times$ and let $g_{\Delta,e}\in O(N)$ be defined as in \eqref{Eq:gre}. 
Note that $e_i, 1 \leq i\leq t,$ are units of $\Z_{k_i}$, i.e., $\gcd(k_i,e_i)=1$ for all $1\le i\le t$. 
By Lemma \ref{L:fpf2}, $g_{\Delta,e}$ is fixed-point free and the order of $g_{\Delta,e}$  is equal to the least common multiple of $k_1,\dots,k_t$, say $n$. 

We now assume the following:
\begin{enumerate}[{\rm (i)}]
\item $|N:L|=n$, equivalently, $\chi_\Delta\in (1/n)N^*$ (see Lemma \ref{L:indexn});
\item $g_{\Delta,e}\in O(L)$, equivalently, $\lambda_e=(\lambda_{e_1}^1,\dots,\lambda_{e_t}^t)\in N^*$ (see Lemma \ref{L:gNc}).
\end{enumerate}
Set $\g_i=(V_{R_i})_1\subset (V_{N})_1$.
Let $\hat{g}_{\Delta,e}\in O(\hat{N})$ be a (standard) lift of $g_{\Delta,e}\in O(N)$.
By Lemma \ref{L:standardlift} (2), the order of $\hat{g}_{\Delta,e}$ is also $n$.
In addition, the restriction of $\hat{g}_{\Delta,e}$ to $V_R$ is also a standard lift of $g_{\Delta,e}\in O(R)$.
By Lemma \ref{L:standardlift} (1), we may assume that the restriction of $\hat{g}_{\Delta,e}$ to $V_{R_i}$ is $\hat{g}_{\g_i}$ up to conjugation.
Here the inner automorphism $\hat{g}_{\g_i}$ of $V_{R_i}$ is defined in Sections \ref{S:autoLie} and \ref{S:ext}. 

Let $\varphi_{\g_i}$ and $\zeta_{\g_i}$ be the inner automorphisms of $V_{R_i}$ defined in Sections \ref{S:autoLie} and \ref{S:ext}. 
Then $\varphi=\bigotimes_{i=1}^t\varphi_{\g_i}$ and $\zeta_0=\bigotimes_{i=1}^t\zeta_{\g_i}$ are inner automorphisms of $V_{{R}}\cong\bigotimes_{i=1}^tV_{R_i}$.
We extend them to inner automorphisms of $V_N$, for which we use the same symbols.
Note that 
\begin{equation*}\label{phi_c}
\varphi=\sigma_{-\chi_{\Delta}}=\exp(2\pi\sqrt{-1}({\chi}_\Delta)_{(0)})
\end{equation*}
by \eqref{Eq:varphi} and that $V_N^\varphi=V_L$.  
By Assumption (i),  the order of $\varphi$ is $n$ on $V_N$.
In addition, $\varphi$ and $\hat{g}_{\Delta,e}$ are still commutative on $V_{N}$; indeed, by \eqref{Eq:grhoc}, 
$$\hat{g}_{\Delta,e}\varphi\hat{g}_{\Delta,e}^{-1}\varphi^{-1}=\exp\left(2\pi\sqrt{-1}(\lambda_e)_{(0)}\right)$$
and it acts trivially on $V_{N}$ by Assumption (ii).

We now slightly modify these automorphisms by inner automorphisms so that the same relations also hold on $V_{N}$.
By Lemma \ref{conjhes} (2), we have $\zeta_0^{-1}\hat{g}_{\Delta,e}\zeta_0=\varphi$ and $\zeta_0^{-1}\varphi\zeta_0={\hat{g}}_{\Delta,e}^{-1}$ on $V_R$.
Since $V_{N}$ is an $(N/R)$-graded simple current extension of $V_R$, we have $$\{\tau\in\Aut(V_N)\mid \tau=id\ {\rm on}\ V_R\}=\{\exp(2\pi\sqrt{-1}x_{(0)})\mid x+N^*\in R^*/N^*\}.$$ 
Hence there exist $x,y\in R^*$ such that $$\zeta_0^{-1}\hat{g}_{\Delta,e}\zeta_0=\varphi\exp(2\pi\sqrt{-1}x_{(0)}),\quad \zeta_0^{-1}\varphi\zeta_0={\hat{g}}_{\Delta,e}^{-1}\exp(2\pi\sqrt{-1}y_{(0)})$$
on $V_{N}$.
Now, we set 
\begin{equation*}\label{hg_zeta}
\hat{g}=\exp(2\pi\sqrt{-1}(-y)_{(0)})\hat{g}_{\Delta,e},\quad {\zeta}=\exp\left(2\pi\sqrt{-1}\left(\sum_{i=1}^{n-1}\frac{i}{n}(g_{\Delta,e})^i(-x+y)\right)_{(0)}\right)\zeta_0.
\end{equation*}
Then $\hat{g}=\hat{g}_{\Delta,e}$ on $V_R$.
In addition, $\hat{g}$ and $\varphi$ are also commutative on $V_{N}$ since $\{\exp a_{(0)}\mid a\in\h\}$ is abelian.
Hence, by the similar calculation as in the proof of \cite[Lemma 4.5]{LS20}, we have 
\begin{equation}
\zeta^{-1}\varphi\zeta={\hat{g}}^{-1}\quad\text{and}\quad \zeta^{-1}\hat{g}\zeta=\varphi\quad \text{on}\quad V_N.\label{Eq:conjsigma}
\end{equation}

\begin{theorem} \label{thm:extra}
Let $t\in\Z_{>0}$.
Let $R_i$ be the root lattice of type $A_{k_i-1}$ for $1\le i\le t$ and set $R=\bigoplus_{i=1}^tR_i$.
Fix a base $\Delta$ of $R(2)$ and let $e\in\bigoplus_{i=1}^t\Z_{k_i}^\times$.
Let $C$ be a subgroup of $\mathcal{D}(R)$ such that $N=L_A(C)$ is even.
Set $L=L_B(C)$.
Assume that Assumptions (i) and (ii) hold.
Let $\hat{g}$ be a lift of the fixed-point free isomety $g_{\Delta,e}$.
Then there exists an automorphism $\zeta$ of $V_{L}^{\hat{g}}$ such that
$V_{L}\circ\zeta\cong V_{N}^{\hat{g}}$ as $V_{L}^{\hat{g}}$-modules.
In particular, $V_{L}^{\hat{g}}$ has an extra automorphism.
\end{theorem}

\begin{proof}
It follows from \eqref{Eq:conjsigma} that ${\zeta}$ preserves $V_N^{\langle\hat{g},\varphi\rangle}$.
Since $\hat{g}$ and $\varphi$ are commuting automorphisms, we have $V_N^{\langle\hat{g},\varphi\rangle}=V_{L}^{\hat{g}}$.
The equations \eqref{Eq:conjsigma} also show $${\zeta}(V_{N}^{\hat{g}})=V_{N}^{\varphi}=V_{L}\quad \text{and}\quad \zeta(V_L^{\hat{g}})=V_L^{\hat{g}},$$ which proves this theorem.
The latter assertion follows from Proposition \ref{P:stabVL1}.
\end{proof}

\begin{remark} If $k_1=\dots=k_t=2$, then $e=(1,1,\dots,1)$, and the Assumptions (i) and (ii) hold;
indeed, evenness of $L_A(C)$ implies that the Hamming weight of any codeword $C$ is a multiple of $4$ and $e\in C^\perp$.
In this case, an extra automorphism of $V_{L_B(C)}^{\tilde{g}}$ was constructed in \cite[Proposition 12.2.5]{FLM} (cf.\ \cite{Sh04}).
\end{remark}

At the end of this subsection, we will rewrite Theorem \ref{thm:extra} when $k_1=\dots=k_t=p$ is an odd prime.
Since $N$ is even, by Lemma \ref{P:so}, $C$ is a self-orthogonal code of length $t$ over $\Z_p$.
Note that $e\in\bigoplus_{i=1}^t\Z_{k_i}^\times$ if and only if the Hamming weight of $e$ is $t$.
In this case, Assumption (i)  always holds and Assumption (ii) holds if $e\in C^\perp$.
Therefore, we obtain the following:

\begin{corollary}\label{C:extra} 
Let $p$ be an odd prime and let $C$ be a self-orthogonal code of length $t$ over $\Z_p$.
Fix a base $\Delta$ of the root system of $L_A(\{\allzero\})\subset L_A(C)$ and let $e$ be a codeword in $C^\perp$ of Hamming weight $t$.
Let $\hat{g}$ be a lift of $g_{\Delta,e}\in O(L_B(C))$.
Then $V_{L_B(C)}^{\hat{g}}$ has extra automorphisms.
\end{corollary}

\section{Case (II): $V_L(1)\circ \tau$ is of twisted type for some $\tau\in  \Aut(V_L^{\hat{g}})$ }\label{S:7}

In this section, we classify even lattices $L$ and isometries $g$ of odd prime order  satisfying Case II in Section \ref{extraauto}.

\subsection{Necessary conditions on even lattices for Case (II)}\label{S:2}
Assume that  there exists $\tau\in  \Aut(V_L^{\hat{g}})$ such that 
\begin{equation}\label{Eq:as62}
V_L(1)\circ \tau\cong V_{L}^T[\hat{g}^s](j),
\end{equation}
for some irreducible $\hat{g}^s$-twisted module $V_{L}^T[\hat{g}^s]$,  $1\leq s\leq |g|-1$,  $0 \leq j\leq |g|-1$ (cf. Section \ref{S:tw}).  
Since $V_L(1)$ is a simple current module, so is $V_{L}^T[\hat{g}^s](j)$.
Hence, the quantum dimension ${\rm qdim}_{V_L} V_{L}^T[\hat{g}^s]$ is $1$ (see \cite[Theorem 2.1 (6)]{AbeLY}). By \cite[Corollary 3.7]{AbeLY}, we also have 
\begin{equation}\label{qdimVT}
({\rm qdim}_{V_L} V_{L}^T[\hat{g}^s] )^2 = p^{-m/(p-1)}|L^*/L|\cdot |L/{R}_L^{{g}^s}|= p^{-m/(p-1)} |L^*/{R}_L^{{g}^s}|,
\end{equation} 
where ${R}_L^{{g}^s}=((1-g^s)L^* )\cap L$.
Recall that $L^*/ (1-g^s)L^*\cong L/(1-g^s)L$ is an elementary abelian $p$-group of order $p^{m/(p-1)}$ (cf. \cite[Lemma A.1]{GL5A}). 
Thus, ${\rm qdim}_{V_L} V_{L}^T[\hat{g}^s]=1$ implies $$|L^*/R_L^{{g}^s}|=p^{m(p-1)}=|L^*/(1-g^s)L^*|.$$
Since $R_L^{{g}^s}\subset (1-g^s)L^*$, we have 
${R}_L^{{g}^s}=(1-g^s)L^*$, i.e., $(1-g^s)L^*\subset L$.
Hence 
\begin{equation}
\mathcal{D}(L)=L^*/L \cong (1-g^s)L^*/ (1-g^s)L \subset L/(1-g^s)L\cong \Z_p^{m(p-1)}.\label{Eq:elemab}
\end{equation}
Since $|g|$ is a prime, we also have $(1-g)L^*\subset L$.
Under the above conditions,  $\irr(V_L^{\hat{g}})$ and the fusion product on it are described in Section \ref{S:1-g}.
In particular, we use the notation $V_{\lambda+L}[\hat{g}^s], \lambda +L\in L^*/L,$ to denote an irreducible $\hat{g}^s$-twisted module as in \eqref{VlL}.

We now discuss possible values for $m$, $p$ and possible structures for $\mathcal{D}(L)$.
Recall from \eqref{Eq:esp} that the conformal weight of $ V_{\lambda+L}[\hat{g}^s]$ is  
\[
\varepsilon= \frac{m(p+1)}{24p}. 
\]

\begin{proposition}\label{DL} 
\begin{enumerate}[{\rm (1)}]
\item $\varepsilon\in\{1-1/p,1\}$.
\item If $\varepsilon=1-1/p$, then $m=\displaystyle\frac{24(p-1)}{p+1}$, $(1-g)L^*=L$ and $|\mathcal{D}(L)|=p^{m/(p-1)}$;
\item If $\varepsilon=1$, then $m=\displaystyle\frac{24p}{(p+1)}$ and $|\mathcal{D}(L)|=p^{m/(p-1)}\left(\displaystyle\frac{m}{p-1}\right)^{-2}$.
\end{enumerate}
\end{proposition}

\begin{proof}
We first note that $t=m/(p-1)\in\Z$. 
Set $M=V_{\lambda+L}[\hat{g}^s](j)$.
By Lemma \ref{Lem:conj} (1) and \eqref{Eq:as62}, we have $\varepsilon(M)=\varepsilon(V_L(1))=1$ and $\dim M_1=\dim(V_L(1))_1=t$.

By Lemma \ref{L:cwtw}, $1\in \{\varepsilon+i/p\mid 0\le i\le p-1\}$.
In addition, if $\varepsilon \leq 1- 2/p$, then by the explicit construction of $V_{\lambda+L}[\hat{g}^s]$ (cf.\ \cite{DL}), we have $\dim M_1\ge (t(t+3)/2)\dim T_{\lambda+L}>t$, which is a contradiction.
Hence $\varepsilon\in\{1-1/p,1\}$, and we obtain (1).

If $\varepsilon=1-1/p$, then by \eqref{Eq:esp} $m=\displaystyle\frac{24(p-1)}{p+1}$.
In addition, $\dim M_1=t\dim T_{\lambda+L}=t$, that is, $\dim T_{{\lambda}+L}=1$. By \eqref{dimT} and \eqref{qdimVT}, we have $(1-g)L^*=L$ and $|\mathcal{D}(L)|=p^{m/(p-1)}$. 
Hence we obtain (2).

If $\varepsilon=1$, then by  \eqref{Eq:esp}, $m=\displaystyle\frac{24p}{p+1}$.
In addition, $\dim M_1=\dim T_{\lambda+L}= t=m/(p-1)$ and $|\mathcal{D}(L)|= p^{m/(p-1)}\left(\displaystyle\frac{m}{p-1}\right)^{-2}$ by \eqref{dimT} and \eqref{qdimVT}.
Hence we obtain (3).
\end{proof}

Recall that $p$ is an odd prime and that $m/(p-1)\in\Z$.
By Proposition \ref{DL}, we have $m<24$.
Hence we obtain the following lemma. 
\begin{lemma}\label{possiblepair}
The possible pairs for $(p,m)$ are $(3,12)$, $(3,18)$, $(5,16)$, $(5,20)$, $(7,18)$, $(11,20)$ and $(23,22)$.
\end{lemma}

\begin{remark}\label{R:Sh04b} The case $p=2$ is discussed in \cite[Proposition 3.14]{Sh04}; the possible values for $m$ in Case (II) are $8$ and $16$.
In addition, $L$ is isometric to $\sqrt{2}E_8$ and the Barnes-Wall lattice of rank $16$, respectively.
\end{remark}

By \eqref{Eq:elemab}, $\mathcal{D}(L)$ is an elementary abelian $p$-group.
For each possible pair $(p,m)$ in Lemma \ref{possiblepair}, $\mathcal{D}(L)$ is determined by Proposition \ref{DL}, which is summarized in Table \ref{parameter}.
\begin{longtable}[c]{|c|c|c|c||c|c|c|} 
\caption{Possible parameters for $L$} \label{parameter}\\
\hline 
$p$  & $m$  & $\mathcal{D}(L)$ &$[t,\dim C]$& $p$  & $m$  & $\mathcal{D}(L)$ \\ \hline \hline 
$3$&$12$&$\Z_3^6$&$[6,1]$& 
$11$&$20$&$\Z_{11}^2$\\ \hline 
$3$&$18$&$\Z_3^5$&$[9,3]$&
$23$&$22$&$\Z_{23}$ \\ \hline 
$5$&$16$&$\Z_5^4$&$[4,1]$&&& \\ \hline 
$5$&$20$&$\Z_5^3$&$[5,2]$&&&\\ \hline 
$7$&$18$&$\Z_7^3$&$[3,1]$ &&& \\ \hline 
\end{longtable}

We now view $\mathcal{D}(L)$ as a vector space over $\F_p=\Z_p(=\Z/p\Z)$. 
Recall that \eqref{Eq:as1} and \eqref{Eq:assumptw} hold.
By Corollary \ref{C:confwt}, for any $\lambda+L\in\mathcal{D}(L)$, we have $\varepsilon(V_{\lambda+L}(0))\in (1/p)\Z$.
Hence $(\lambda|\lambda)/2\in (1/p)\Z$.
Therefore $(\mathcal{D}(L), q)$ forms a non-singular quadratic space over $\F_p$, where the quadratic form $q:\mathcal{D}(L)\to\F_p$ is defined by $$q(\lambda+L) = \frac{p}{2}(\lambda| \lambda) \pmod{p}.$$

\subsubsection{Case: $p=3,5,7$}
Assume that $p=3,5$ or $7$.
By Table \ref{parameter}, we have $\dim \mathcal{D}(L)\ge3$ as a vector space over $\F_p$.
Hence $(\mathcal{D}(L),q)$ contains a non-zero singular vector; it means that there is a $\lambda+L\in\mathcal{D}(L)\setminus\{L\}$ such that $(\lambda|\lambda)\in2\Z$.
By Propositions \ref{LC} and \ref{P:(I)}, there exists a self-orthogonal code $C$ of length $t=m/(p-1)$ over $\Z_p$ such that $L\cong L_B(C)$.

By Lemma \ref{L:discB}, $|\mathcal{D}(L)|=p^2|C^\perp/C|$.
Since $\dim C^\perp+\dim C=t$, we have $|C^\perp/C|=p^{t-2\dim C}$, which determines $\dim C$ (see Table \ref{parameter}).
Since $L$ is rootless, the code $C$ is unique up to equivalence by Lemma \ref{L:uniqueC}.
By Table \ref{T:example}, we obtain the following proposition.

\begin{proposition}\label{P:357} If $p=3,5$ or $7$, then $L\cong L_B(C)$ for some code $C$ given in Lemma \ref{L:uniqueC}.
In particular, if $(p,m)=(3,12)$, $(3,18)$, $(5,16)$, $(5,20)$ or $(7,18)$, then $L$ is isometric to the coinvariant lattice of the Leech lattice $\Lambda$ associated with the conjugacy class $3B$, $3C$, $5B$, $5C$ or $7B$, respectively.
\end{proposition}

\subsubsection{Case: $p=11$}
By Table \ref{parameter} and Proposition \ref{DL}, we have $m={\rm rank}(L)=20$, $\mathcal{D}(L)\cong \F_{11}^2$ and $(1-g)L^*=L$.

\begin{lemma}\label{nosingular}
There are no non-zero singular vectors in $(\mathcal{D}(L),q)$, that is, $(\lambda|\lambda)\notin2\Z$  for any $\lambda+L\in\mathcal{D}(L)\setminus\{L\}$.
In particular, $(\mathcal{D}(L),q)$ is a non-singular $2$-dimensional quadratic space of $(-)$-type.  
\end{lemma}
\begin{proof}
Suppose that there exists $\lambda+L\in \mathcal{D}(L)\setminus\{L\}$ such that $(\lambda|\lambda)\in2\Z$. 
Then ${\rm Span}_\Z\{\lambda,L\}$ is even unimodular of rank $20$, which is impossible. 
\end{proof}

\begin{lemma}\label{L:11} Let $U$ be an even unimodular lattice of rank $24$.
If $U$ has an isometry $f$ of order $11$ such that both $U^f$ and $U_f$ are rootless, then $U$ is isometric to the Leech lattice.
\end{lemma}
\begin{proof} By the orders of the automorphism groups of even unimodular lattices of rank $24$ (cf.\ \cite[Chapter 16]{CS}), the root system of $U(2)$ is one of $\emptyset$, $A_1^{24}$, $A_2^{12}$, $A_{11}D_7E_6$, $A_{12}^2$, $A_{15}D_9$, $A_{17}E_7$, $D_{12}^2$, $A_{24}$, $D_{16}E_8$ and $D_{24}$.
If $U(2)\neq\emptyset$, namely, $U\not\cong\Lambda$, then one can easily verify that $U^f$ or $U_f$ contains roots, which contradicts the assumption.
\end{proof}

\begin{proposition}\label{P:11A}
If $p=11$, then $L$ is isometric to the coinvariant lattice of the Leech lattice $\Lambda$ associated with the conjugacy class $11A$ of $O(\Lambda)$. 
\end{proposition}

\begin{proof}
Let $h\in O(\Lambda)$ be an element in the conjugacy class $11A$. Then $\Lambda^h$ has rank $4$ and $(\mathcal{D}(\Lambda^h),q)$ is a $2$-dimensional non-singular space over $\F_{11}$ of $(-)$-type (\cite{HL90}). 
Now consider $X= L \perp \Lambda^h$. Then $\rank X=24$.
By Lemma \ref{nosingular}, $(\mathcal{D}(X),q)\cong (\mathcal{D}(L),q)\oplus (\mathcal{D}(\Lambda^h),q)$ is a $4$-dimensional non-singular quadratic space over $\F_{11}$ of $(+)$-type. 
Hence, there exists an even unimodular lattice $U\supset L\perp \Lambda^h$. 
It follows from $(1-g)L^*=L$ that $g$ can be extended to an isometry of $U$ by acting trivially on $\Lambda^h$. 
Note that $\Lambda^h=U^g$ and $L=U_g$, which shows that $U$ has an isometry $g$ of order $11$ such that both $U^g$ and $U_g$ are rootless. 
By Lemma \ref{L:11}, we have $U\cong \Lambda$.
Since $O(\Lambda)$ has only one conjugacy class of order $11$, $g$ is conjugate to $h$, and $L\cong \Lambda_h$. 
\end{proof}

\subsubsection{Case: $p=23$}
By Table \ref{parameter} and Proposition \ref{DL}, we have $m={\rm rank}(L)=22$, $\mathcal{D}(L)\cong \F_{23}$ and $(1-g)L^*=L$.

\begin{lemma}\label{L:-123}
There exists $\lambda+L\in \mathcal{D}(L)$ such that $q(\lambda+L)=-1$.   
\end{lemma}

\begin{proof}
Suppose that the assertion is false.
Since $-1$ is a non-square in $\F_{23}$, there exists $\lambda+L \in \mathcal{D}(L)$ such that
$q(\lambda+L)=1$.
Let $L'$ be the root lattice of type $A_{22}$.
Then $\mathcal{D}(L')\cong\F_{23}$.
By \eqref{Eq:innerij}, we have $q(\lambda_5+L')=-1$ for $\lambda_5+L'\in\mathcal{D}(L')$.
Then $X=L\perp L'$ is an even lattice of rank $44$ and $\mathcal{D}(X)\cong  \F_{23}^2$.
It follows from $q((\lambda,\lambda_5)+X)=0$ that $(\lambda,\lambda_5)+X$ has even norm.
Hence ${\rm Span}_\Z\{X,(\lambda,\lambda_5)\}$ is even unimodular of rank $44$, which is impossible. 
\end{proof}

\begin{lemma}\label{L:232} Let $U$ be an even unimodular lattice of rank $24$.
If $U$ has an isometry $f$ of order $23$ such that both $U^f$ and $U_f$ are rootless, then $U$ is isometric to the Leech lattice.
\end{lemma}
\begin{proof} By the orders of the automorphism groups of even unimodular lattices of rank $24$ (cf.\ \cite[Chapter 16]{CS}), the root system of $U(2)$ is one of $\emptyset$, $A_1^{24}$, $A_{24}$ and $D_{24}$.
If $U(2)\neq\emptyset$, namely, $U\not\cong\Lambda$, then one can easily verify that $U^f$ or $U_f$ contains roots, which contradicts the assumption.
\end{proof}

The following is well-known (cf.\ \cite{ATLAS}).
\begin{lemma}\label{L:23} The automorphism group of the Leech lattice has exactly two conjugacy classes $23A$ and $23B$ of order $23$, and the cyclic subgroups generated by elements in $23A$ and $23B$ are conjugate.
In particular, the coinvariant lattice of $\Lambda$ associated with an isometry of order $23$ is unique up to isometry.
\end{lemma}

\begin{proposition}\label{P:23A}
If $p=23$, then $L$ is isometric to the coinvariant lattice of  the Leech lattice $\Lambda$ associated with the conjugacy class $23A$.
\end{proposition}

\begin{proof}
Let $h\in O(\Lambda)$ be an element in the conjugacy class $23A$.
Then $\Lambda^h$ is an even lattice of rank $2$ with Gram matrix $\begin{pmatrix} 4 &1\\ 1&6\end{pmatrix}$ and $\mathcal{D}(\Lambda^h)\cong\F_{23}$ (\cite{HL90}).
In addition, there exists $\mu+\Lambda^h\in\mathcal{D}(\Lambda^h)$ such that $q(\mu+\Lambda^h)=1$.
Set $X=L\perp \Lambda^h$.
Then $\mathcal{D}(X)\cong\F_{23}^2$.
By Lemma \ref{L:-123}, there exists $\lambda+L\in \mathcal{D}(L)$ such that $q((\lambda,\mu)+X)=0$.
Hence $(\lambda,\mu)+X$ has even norm, and $U={\rm Span}_\Z\{X,(\lambda,\mu)\}$ is an even unimodular lattice of rank $24$.
It follows from $(1-g)L^*=L$ that $g$ can be extended to an isometry of $U$ by acting trivially on $\Lambda^h$.
Note that $U^g=\Lambda^h$ and $U_g=L$, which shows that $U$ has an isometry of order $23$ such that both $U^g$ and $U_g$ are rootless.
By Lemma \ref{L:232}, we have $U\cong\Lambda$.
By Lemma \ref{L:23}, $\langle g\rangle$ is conjugate to $\langle h\rangle$, and $L\cong \Lambda_h$.
\end{proof}

\subsection{Extra automorphisms and coinvariant lattices of the Leech lattice}\label{S:3}
In this subsection, we prove that the possible rootless even lattices in the previous section actually satisfy the condition in Case (II).

The following proposition is a slight modification of a theorem in \cite{LamCFT}.

\begin{proposition}\label{P:extw} Let $U$ be a rootless even unimodular lattice.
Let $h\in O(U)$ and let $\hat{h}\in O(\hat{U})$ be a standard lift of $h$.
Assume that $\hat{h}$ has type $0$; let $V_U^{{\rm orb}(\hat{h})}$ be the holomorphic VOA obtained by  the orbifold construction associated with $U$ and $\hat{h}$ (see \cite[Section 5]{EMS} for the definitions).
Assume that $V_U^{{\rm orb}(\hat{h})}\cong V_U$. 
We also assume that the conjugacy class of the cyclic group $\langle \hat{h}\rangle$ in $\Aut(V_U)$ is uniquely determined by $|\hat{h}|$ and the VOA structure of $V_U^{\hat{h}}$.
Then there exists $\tau\in\Aut(V_{U_h}^{\hat{h}})$ such that $V_{U_h}(1)\circ\tau$ is of twisted type.
\end{proposition}
\begin{proof} Set $n=|\hat{h}|$ and $W=V_U^{{\rm orb}(\hat{h})}$.
Let $\psi:W\to V_U$ be an isomorphism of VOAs.
Since $W$ is a $\Z_n$-graded simple current extension of $V_U^{\hat{h}}$, we obtain an order $n$ automorphism $\xi$ of $W$ associated with the $\Z_n$-grading.
Note that $W^\xi=V_U^{\hat{h}}$ and that $W^{{\rm orb}(\xi)}\cong V_U$.

Set $\tilde{\xi}=\psi \xi\psi^{-1}$.
Then $\tilde{\xi}\in\Aut(V_U)$.
Clearly, $|\tilde{\xi}|=|\xi|=n$ and $V_U^{\tilde{\xi}}\cong V_U^{\hat{h}}$.
In addition, $V_U^{{\rm orb}(\tilde{\xi})}\cong W^{{\rm orb}(\xi)}\cong V_U$.
By the assumption, $\langle \tilde{\xi}\rangle$ is conjugate to $\langle\hat{h}\rangle$ in $\Aut(V_U)$; let $\kappa\in\Aut(V_U)$ such that $\kappa\langle\tilde{\xi}\rangle \kappa^{-1}=\langle \hat{h}\rangle$.
Then $$\kappa\psi(V_U^{\hat{h}})=\kappa\psi(W^\xi)=\kappa(V_U^{\tilde{\xi}})=V_U^{\hat{h}}.$$ 
Moreover $\kappa\psi$ sends $\Com_{V_U^{\hat{h}}}((V_U^{\hat{h}})_1)\cong V_{U_h}^{\hat{h}}$ to itself; we have $\kappa\psi\in \Aut(V_{U_h}^{\hat{h}})$ by restriction.
Here $\Com_{V_U^{\hat{h}}}((V_U^{\hat{h}})_1)$ is the commutant in $V_U^{\hat{h}}$ of $(V_U^{\hat{h}})_1$ (\cite{FZ}).
Since $\kappa\psi$ induces a $V_U^{\hat{h}}$-module isomorphism from $V_U^{{\rm orb}(\hat{h})}$ to $V_U$, we have $V_U^{{\rm orb}(\hat{h})}\cong V_U\circ (\kappa\psi)$ as $V_U^{\hat{h}}$-modules.
We now view $V_U$ and $V_U^{{\rm orb}(\hat{h})}$ as modules over $V_{U^h}\otimes V_{U_h}^{\hat{h}}(\subset V_U^{\hat{h}})$.
Note that the irreducible $V_{U_h}^{\hat{h}}$-module $V_{U_h}(1)$ appears as a submodule of $V_U(1)$, the eigenspace of $\hat{h}$ with eigenvalue $\exp(2\pi\sqrt{-1}/n)$.
Since $(\kappa\psi)^{-1}$ sends $V_U(1)$ to some irreducible $V_U^{\hat{h}}$-submodule of $V_U^{{\rm orb}(\hat{h})}$ of twisted type, $(\kappa\psi)^{-1}(V_U(1))$ contains only irreducible $V_{U_h}^{\hat{h}}$-modules of twisted type.
Therefore we obtain this proposition.
\end{proof}

We give a sufficient condition for the hypothesis of the proposition above.

\begin{lemma}\label{L:unicon} Let $K$ be a rootless even lattice.
Let $h\in O(K)$ of odd prime order $p$ such that $K^h\neq0$.
Assume that the conjugacy class of $\langle h\rangle$ is uniquely determined by the order of $h$ and $\rank K^h$.
Let $\hat{h}\in O(\hat{K})$ be a standard lift of $h$ and let $\xi\in \Aut(V_K)$ such that $|\xi|=|\hat{h}|$ and $V_K^{\hat{h}}\cong V_K^{\xi}$.
Then $\langle \xi\rangle$ is conjugate to $\langle \hat{h}\rangle$ in $\Aut(V_K)$.
\end{lemma}
\begin{proof} Since the order of $h$ is odd, we have $|\xi|=|\hat{h}|=|h|=p$ by Lemma \ref{L:standardlift} (2).
In addition, $\dim (V_K^{\xi})_1=\dim (V_K^{\hat{h}})_1=\rank K^h$.
By \cite[Lemma 4.7 (2)]{LS20}, there exist $z\in O({K})$ and $x\in \Q\otimes_\Z K^z$ such that $\xi=\sigma_x\hat{z}$.
Then $\dim (V_K^\xi)_1=\rank K^z=\rank K^h\neq0$.
Hence $z\neq id$ and $|z|=p=|h|$ since $p$ is a prime.
By the assumption, $z$ is conjugate to $h^i$ in $O(K)$ for some $i$.
Note that $K^z\cong K^h$. 

Suppose, for a contradiction, that $\sigma_x\neq id$.
By \cite[Lemma 4.5 (2)]{LS20}, $x\notin (K^z)^*$.
Then $\Com_{V_K^{\xi}}(\Com_{V_K^\xi}((V_K^{\xi})_1))= V_{S}$, where $S=\{v\in K^z\mid (v| x)\in\Z\}$.
Note that $S\not\cong K^z$ because $x\notin (K^z)^*$.
On the other hand, $\Com_{V_K^{\hat{h}}}(\Com_{V_K^{\hat{h}}}((V_K^{\hat{h}})_1))= V_{K^h}$.
The assumption $V_\Lambda^{\hat{h}}\cong V_\Lambda^{\xi}$ implies $V_S\cong V_{K^h}$ and $S\cong K^h$, which is a contradiction.

Therefore, $\sigma_x=id$ and $\xi=\hat{z}$.
By Lemma \ref{L:standardlift}, $\xi$ and $\hat{h^i}$ are conjugate, and $\langle\hat{z}\rangle$ and $\langle\hat{h}\rangle$ are also conjugate.
\end{proof}

\begin{lemma}\label{L:orbLeech} Let $\Lambda$ be the Leech lattice and let $h\in O(\Lambda)$ such that the conjugacy class of $h$ is $3B,3C,5B,5C,7B,11A$ or $23A$ (as the notations in \cite{ATLAS}).
\begin{enumerate}[{\rm (1)}]
\item The conjugacy class of the cyclic group $\langle h\rangle$ in $O(\Lambda)$ is uniquely determined by the order $|h|$ and $\rank \Lambda^h$.
\item Let $\hat{h}\in O(\hat{\Lambda})$ be a standard lift of $h$.
Then $V_\Lambda^{{\rm orb}(\hat{h})}\cong V_\Lambda$.
\end{enumerate}
\end{lemma}
\begin{proof} The assertion (1) can be verified by \cite{ATLAS} (cf.\ \cite{HL90}).

Set $p=|h|$.
Since $p$ is an odd prime, the order of $\hat{h}$ is also $p$ by Lemma \ref{L:standardlift} (2).
Let $1\le s\le p-1$.
For the explicit construction of the irreducible $\hat{h}^s$-twisted $V_\Lambda$-module $V_\Lambda[\hat{h}^s]$, see \cite{DL} (cf. \cite[Section 4.3]{LS20}).
Let $\varepsilon$ be the conformal weight of $V_\Lambda[\hat{h}^s]$.

If the conjugacy class of $h$ is $3B,5B,7B,11A$ or $23A$, then $\varepsilon=1-1/p$.
Since the minimum norm of $(\Lambda^h)^*$ is $4/p$, 
$\dim (V_\Lambda[\hat{h}^s])_1$ is equal to the dimension of the eigenspace of $h^s$ with eigenvalue $\exp(2\pi\sqrt{-1}/p)$, which is also equal to $(\rank \Lambda_h)/(p-1)$.

If the conjugacy class of $h$ is $3C$ or $5C$, then $\varepsilon=1$.
Then $\dim (V_\Lambda[\hat{h}^s])_1$ is equal to $|\Lambda_h/(1-h^s)\Lambda_h|^{1/2}$. 
By the proof of Proposition \ref{DL} (3) (cf.\ Lemma \ref{L:uniqueC}), we have $\dim (V_\Lambda[\hat{h}^s])_1=(\rank \Lambda_h)/(p-1)$.

Hence, we have
$$\dim (V_\Lambda^{{\rm orb}(\hat{h})})_1=\dim (V_\Lambda^{\hat{h}})_1+\sum_{i=1}^{p-1}(V_\Lambda[\hat{h}^i])_1=\rank \Lambda^h+(p-1)\frac{\rank \Lambda_h}{p-1}=24.$$
By \cite{DMb}, we have $V_\Lambda^{{\rm orb}(\hat{h})}\cong V_\Lambda$.
\end{proof}

Combining Proposition \ref{P:extw} and Lemmas \ref{L:unicon} and \ref{L:orbLeech}, we obtain the following:

\begin{proposition}\label{P:extwp} Let $h\in O(\Lambda)$ whose conjugacy class is $3B,3C,5B,5C,7B,11A$ or $23A$.
Then there exists $\tau\in\Aut(V_{\Lambda_h}^{\hat{h}})$ such that $V_{\Lambda_h}(1)\circ\tau$ is of twisted type.
\end{proposition}

\subsection{Classification results for Case (II)}
We summarize the results in Sections \ref{S:2} and \ref{S:3} as follows (see Lemmas \ref{L:uniqueC} and \ref{possiblepair}, Propositions \ref{P:357}, \ref{P:11A}, \ref{P:23A} and \ref{P:extwp} and Remark \ref{R:Sh04b}):

\begin{theorem}\label{T:(II)} Let $L$ be a rootless even lattice of rank $m$.
Let $g$ be a fixed-point free isometry of $L$ of prime order $p$ and let $\hat{g}$ be a lift of $g$.
Then there exists an automorphism $\tau$ of $V_L^{\hat{g}}$ such that the $\tau$-conjugate $V_L(1)\circ \tau$ of $V_L(1)$ is of twisted type if and only if one of the following holds:
\begin{enumerate}[{\rm (1)}]
\item $p=2$ and $L$ is isometric to $\sqrt2E_8$ or the Barnes-Wall lattice of rank $16$;
\item $p=3$ and $L\cong L_B(C)$, where $C=\langle (1,1,1,1,1,1)\rangle_{\Z_3}$ or $C$ is the code over $\Z_3$ with the generator matrix $\begin{pmatrix}1&1&1&1&1&1&1&1&1\\ 1&1&1&-1&-1&-1&0&0&0\\
1&-1&0&1&-1&0&1&-1&0\end{pmatrix}$;
\item $p=5$ and $L\cong L_B(C)$, where $C=\langle (1,1,2,2)\rangle_{\Z_5}$ or $C$ is the code over $\Z_5$ with the generator matrix 
$$\begin{pmatrix}1&1&1&1&1\\ 1&2&4&3&0\end{pmatrix};$$
\item $p=7$ and $L\cong L_B(C)$, where $C=\langle (1,2,3)\rangle_{\Z_7}$;
\item $p=11$ and $L$ is isometric to the coinvariant lattice $\Lambda_{11A}$ of the Leech lattice associated with the conjugacy class $11A$ of $O(\Lambda)$;
\item $p=23$ and $L$ is isometric to the coinvariant lattice $\Lambda_{23A}$ of the Leech lattice associated with the conjugacy class $23A$ of $O(\Lambda)$.
\end{enumerate}
\end{theorem}
\begin{remark}
If $p=3$ (resp. $p=5$) and $C$ is equivalent to $\langle (1,1,1,1,1,1)\rangle_{\Z_3}$ (resp. $\langle (1,1,2,2)\rangle_{\Z_5}$), then $L_B(C)$ is isometric to the Coxeter-Todd lattice $K_{12}$ of rank $12$ (resp. $\Lambda_{5B}$).
Then $\Aut(V_{L_B(C)}^{\hat{g}})$ is determined in \cite{CLS} (resp. \cite{Lam18b}) and it actually has an extra automorphism $\tau$ such that $V_{L_B(C)}(1)\circ\tau$ is of twisted type.
\end{remark}

As we mentioned in Section 5, an extra automorphism of $V_L^{\hat{g}}$ satisfies Case (I) or (II).
Combining Remark \ref{R:Sh04a}, Proposition \ref{P:(I)}, Corollary \ref{C:extra}, Theorem \ref{T:(II)}, we have proved Main theorem 1 in Introduction.
We also obtain Main corollary 1 by Main theorem 1 and Definition \ref{Def:extra}.
In addition, we obtain Main corollary 2 from Table \ref{T:example}, Remark \ref{R:2A} and Theorem \ref{T:(II)}.

\begin{remark} In this article, we classify even lattices $L$ and fixed-point free isometries $g$ so that $V_L^{\hat{g}}$ has extra automorphisms.
The group structure of $\Aut(V_L^{\hat{g}})$ could be determined by using similar arguments as in \cite[Section 5]{Sh06} or \cite{BLS} if the group structure of $O(L)$  is explicitly known.
\end{remark}


\begin{thebibliography}{100}



\bibitem[ALY18]{AbeLY}
T.\ Abe, C.H.\ Lam and H.\ Yamada, Extensions of tensor products of $\Z_p$-orbifold models of the lattice vertex operator algebra $V_{\sqrt2A_{p-1}}$, \emph{J. Algebra} {\bf 510} (2018), 24--51.

\bibitem[BK04]{BK04}B. Bakalov, and V.\,G.\,Kac, Twisted modules over lattice vertex algebras, 
\emph{Lie theory and its applications in physics V}, 3--26, World Sci. Publ., River Edge, NJ, 2004.

\bibitem[BLS22+]{BLS} K. Betsumiya, C.H. Lam and H. Shimakura, Automorphism groups of cyclic orbifold vertex operator algebras associated with the Leech lattice and some non-prime isometries, to appear in \emph{Israel Journal of Mathematics}; arXiv:2105.04191.

\bibitem[Bo86]{Bo}
R.E.\ Borcherds, Vertex algebras, Kac-Moody algebras, and the Monster, {\it Proc.\ Nat'l.\ Acad.\ 
Sci.\ U.S.A.} {\bf 83} (1986), 3068--3071.

\bibitem[BCP97]{MAGMA}
W.\ Bosma, J.\ Cannon and C.\ Playoust, The Magma algebra system I: The user
language, \emph{J. Symbolic Comput.} {\bf 24} (1997), 235--265.

\bibitem[CM]{CM} S.\ Carnahan and M.\ Miyamoto, 
Regularity of fixed-point vertex operator subalgebras; arXiv:1603.05645. 

\bibitem[CLS18]{CLS} 
H.Y.\ Chen, C.H.\ Lam and H.\ Shimakura, On $\Z_3$-orbifold construction of the Moonshine vertex operator algebra, Math. Z. \textbf{288} (2018), no. 1-2, 75-100.


\bibitem[ATLAS]{ATLAS}
J.H. Conway, R.T. Curtis, S.P. Norton, R.A. Parker and  R.A. Wilson,
ATLAS of finite groups. 
Clarendon Press, Oxford, 1985.

\bibitem[CS99]{CS}
J.H.\ Conway and N.J.A.\ Sloane, Sphere packings, lattices and groups, 3rd Edition, Springer, New 
York, 1999.

\bibitem[Do93]{dong}
C.\ Dong, Vertex algebras associated with even lattices, {\it J. Algebra} {\bf 161} (1993), 245--265.

\bibitem[DJX13]{DJX}
C. Dong, X. Jiao and F. Xu, Quantum dimensions and quantum Galois theory. {\em Trans. Amer. Math. 
Soc.} {\bf 365} (2013), 6441-6469.

\bibitem[DL96]{DL} C.\ Dong and J.\ Lepowsky,
The algebraic structure of relative twisted vertex operators,
{\it J. Pure Appl. Algebra} {\bf 110} (1996), 259--295.

\bibitem[DLM00]{DLM2}
C.\ Dong, H.\ Li, and G.\ Mason, Modular-invariance of trace functions
in orbifold theory and generalized Moonshine, 
\emph{Comm. Math. Phys.} {\bf 214} (2000), 1--56.

\bibitem[DM04]{DMb}
C.\ Dong and G.\ Mason, Holomorphic vertex operator algebras of small central charge, {\it Pacific 
J. Math.} {\bf 213} (2004), 253--266.

\bibitem[DN99]{DN}
C.\ Dong and K.\ Nagatomo, 
Automorphism groups and twisted modules for lattice vertex operator algebras, {\it in} Recent 
developments in quantum affine algebras and related topics (Raleigh, NC, 1998), 117--133,
{\it Contemp. Math.}, {\bf 248}, Amer. Math. Soc., Providence, RI, 1999. 

\bibitem[DRX17]{DRX}
C.\ Dong, L.\ Ren and F.\ Xu, On Orbifold Theory, \emph{Adv. Math.} {\bf 321} (2017), 1--30.

\bibitem[EMS20]{EMS}
J. van Ekeren, S.\ M\"oller and N.\ Scheithauer, Construction and classification of
holomorphic vertex operator algebras,
\emph{J. Reine Angew. Math.}, \textbf{759} (2020), 61--99.
 
\bibitem[FHL93]{FHL}
I.B.\ Frenkel, Y.\ Huang and J.\ Lepowsky, On axiomatic approaches to vertex operator algebras and 
modules,  {\it Mem. Amer. Math. Soc.} {\bf 104} (1993), viii+64 pp.

\bibitem[FLM88]{FLM}
I.\ Frenkel, J.\ Lepowsky and A.\ Meurman, Vertex operator algebras and the Monster, Pure and Appl.\ Math., Vol.134, Academic Press, Boston, 1988.

\bibitem[FZ92]{FZ}
I.\ Frenkel and Y.\ Zhu, Vertex operator algebras associated to representations
of affine and Virasoro algebras, {\it Duke Math. J.} {\bf 66} (1992),
123--168.

\bibitem[Gr98]{Gr1} R.L.\ Griess, A vertex operator algebra related to $E\sb 8$ with 
automorphism group ${\rm O}\sp +(10,2)$, {\it Ohio State Univ. Math. Res. Inst. Publ.} {\bf 7} 
(1998), 43--58.

\bibitem[GL11]{GL5A}  R.\,L.\,Griess, Jr. and C.\,H.\,Lam, A moonshine path for $5A$ node  and associated lattices of ranks 8 and 16, \emph{J. Algebra}, 331(2011), 338-361.

\bibitem[HaL90]{HL90} K.\ Harada and M.L.\ Lang, On some sublattices of the Leech lattice, \emph{Hokkaido Math. J.} {\bf 19} (1990), 435--446. 


\bibitem[HuL95]{HL}
Y.Z.\ Huang and J.\ Lepowsky, A theory of tensor product for module category of a vertex operator algebra, III, {\it J. Pure Appl. Algebra}, {\bf 100} (1995), 141--171.

\bibitem[La20a]{Lam19} C.H. Lam, Cyclic orbifold of lattice vertex operator algebras having group-like fusions, \emph{Lett. Math. Phys.} {\bf 110} (2020), 1081–1112.

\bibitem[La20b]{Lam18b} C.H. Lam, Automorphism group of an orbifold vertex operator algebra associated with the Leech lattice, \emph{in} Vertex operator algebras, number theory and related topics, 127–138, \emph{Contemp. Math.}, {\bf 753}, Amer. Math. Soc., Providence, RI, 2020.

\bibitem[La]{LamCFT} C.H. Lam,  Some observations about the automorphism groups of certain orbifold vertex operator algebras, to appear in RIMS Kôkyûroku Bessatsu.

\bibitem[LS20]{LS20} C.H.\ Lam and H.\ Shimakura, On orbifold constructions associated with the Leech lattice vertex operator algebra, \emph{Math. Proc. Cambridge Philos. Soc.} {\bf 168} (2020), 261-–285.

\bibitem[LY14]{LY2} C.H. Lam and H. Yamauchi, On $3$-transposition groups generated by 
$\tau$-involutions associated to $c=4/5$ Virasoro vectors, \emph{J. Algebra}, 416 (2014), 84-121. 


\bibitem[Le85]{Le} J.\ Lepowsky, Calculus of twisted vertex operators,  {\it Proc.
Natl. Acad. Sci. USA} {\bf 82} (1985), 8295--8299.


\bibitem[Mc22]{Mc22} R. McRae, On semisimplicity of module categories for finite non-zero index vertex operator subalgebras, \emph{Lett. Math. Phys.} \textbf{112}, No. 2, Paper No. 25, 28 p. (2022)

\bibitem[Mi15]{Mi}
M.\ Miyamoto, $C_2$-cofiniteness of cyclic-orbifold models, {\it Comm. Math. Phys.} {\bf 335} (2015),  1279--1286.


\bibitem[Sh04]{Sh04}
H.~Shimakura, The automorphism group of the vertex operator algebra $V_L^+$ for an even lattice $L$ without roots, {\it J. Algebra} {\bf 280} (2004), 29--57.

\bibitem[Sh06]{Sh06}
H.~Shimakura, The automorphism groups of the vertex operator algebras $V_L^+$: general case, {\it Math. Z.} {\bf 252} (2006), 849--862.

\bibitem[Sh07]{Sh07}
H.~Shimakura, Lifts of automorphisms of vertex operator algebras in
  simple current extensions, \emph{Math. Z.} \textbf{256} (2007), no.~3, 491--508.


\end{thebibliography}
\end{document}